\documentclass[10pt,a4paper]{article}
\usepackage[T1]{fontenc}
\usepackage[utf8]{inputenc}
\usepackage{amsmath}
\usepackage{amsfonts}
\usepackage{amssymb}
\usepackage{amsthm}
\usepackage{wasysym}
\usepackage{hyperref}
\usepackage{etoolbox}
\usepackage[normalem]{ulem}
\usepackage{xcolor}
\usepackage{arydshln}
\usepackage{graphicx}
\usepackage{arydshln}
\usepackage{appendix}
\usepackage{epstopdf}
\usepackage{todonotes}
\usepackage{subcaption}
\usepackage{epstopdf}
\usepackage{graphicx}
\usepackage[left=90.00pt, right=90.00pt, top=100.00pt, bottom=100.00pt]{geometry}

\newtheorem{assumption}{Assumption}
\newtheorem{theorem}{Theorem}
\newtheorem{lemma}{Lemma}

\theoremstyle{definition}
\newtheorem{example}{Example}
\newtheorem{definition}{Definition}
\newtheorem{property}{Property}
\newtheorem{proposition}{Proposition}

\theoremstyle{remark}
\newtheorem{remark}{Remark}
\DeclareMathOperator{\rank}{rank}
\DeclareMathOperator{\dist}{dist}
\DeclareMathOperator{\modolo}{mod}
\DeclareMathOperator{\argmin}{argmin}
\newcommand{\minus}{\scalebox{0.75}[1.0]{\( - \)}}
\newcommand{\plus}{\scalebox{0.75}[0.75]{\( + \)}}

\title{A pseudo-spectra based characterisation of the robust strong H-infinity norm of time-delay systems with real-valued and structured uncertainties}
\date{}
\author{Pieter Appeltans and Wim Michiels}
	
\begin{document}
	\maketitle
	\begin{abstract}

This paper examines the robust (strong) H-infinity norm of a linear time-invariant system with discrete delays. The considered system is subject to real-valued, structured, Frobenius norm bounded uncertainties on the coefficient matrices. The robust H-infinity norm is the worst case value of the H-infinity norm over the realisations of the system and hence an important measure of robust performance in control engineering. However this robust H-infinity norm is a fragile measure, as for a particular realization of the uncertainties the H-infinity norm might be sensitive to arbitrarily small perturbations on the delays. Therefore, we introduce the robust strong H-infinity norm, inspired by the notion of strong stability of delay differential equations of neutral type, which takes into account both the perturbations on the system matrices and infinitesimal small delay perturbations. This quantity is a continuous function of the nominal system parameters and delays. The main contribution of this work is the introduction of a relation between this robust strong H-infinity norm and the the pseudo-spectrum of an associated singular delay eigenvalue problem. This relation is subsequently employed in a novel algorithm for computing the robust strong H-infinity norm of uncertain time-delay systems. Both the theoretical results and the algorithm are also generalized to systems with uncertainties on the delays, and systems described by a class of delay differential algebraic equations. 
	\end{abstract}

	\section{Introduction}
	\label{sec:introduction}
	
	In this work we will focus on  linear, time-invariant systems with discrete delays: 
	\begin{equation}
	\label{eq:nominal_system_description}
	\left\{\begin{array}{rcl}
	\dot{x}(t) & = & A_0~x(t) + \sum_{k=1}^{K} A_k~x(t-\tau_k) + B_0~w(t) + \sum_{k=1}^{K} B_{k}~w(t-\tau_k)  \\ [0.7em]
	z(t) & = & C_0~x(t) + \sum_{k=1}^{K}C_k~x(t-\tau_k) + D_0~w(t) + \sum_{k=1}^{K} D_{k}~w(t-\tau_k)
	\end{array}\text{.}\right. 
	\end{equation} 
	with $x \in \mathbb{R}^{n}$ the state, $w\in\mathbb{R}^{m}$ the exogenous input and $z\in\mathbb{R}^{p}$ the exogenous output; $A_k$, $B_k$, $C_k$ and $D_k$ for $k=0,\dots,K$ real-valued system matrices of appropriate dimension; and $\vec{\tau} = (\tau_1,\tau_2,\dots,\tau_K) \in (\mathbb{R}^{+})^{K}$ discrete delays. The transfer function associated with this system equals:	
	\begin{equation*}
	\begin{array}{lc}
	T(s;\vec{\tau}) = & \left(C_0+\sum\limits_{k=1}^{K}C_k e^{-s\tau_k}\right)\left(sI-A_0-\sum\limits_{k=1}^{K}A_{k}e^{-s\tau_k}\right)^{-1}\!\left(B_0+\sum\limits_{k=1}^{K} B_k e^{-s \tau_k}\right) \\ &+ D_0+\sum\limits_{k=1}^{K}D_{k}e^{-s \tau_k} \text{.}
	\end{array}
	\end{equation*}
	
	The $\mathcal{H}_{\infty}$-norm is an important performance measure of dynamical systems as it quantifies the disturbance rejection of the system. It is frequently used in the robust control framework \cite{Zhou1998}. For system \eqref{eq:nominal_system_description}, if exponentially stable, the $\mathcal{H}_{\infty}$-norm is equal to the supremum of the frequency response (ie. the transfer function evaluated at the imaginary axis) measured in spectral norm:
	\begin{equation}
	\label{eq:definition_hinf}
	\|T(\cdot;\vec{\tau})\|_{\mathcal{H}_\infty} = \sup_{\omega \in \mathbb{R}} \sigma_1\left(T(\jmath\omega;\vec{\tau})\right)
	\end{equation}
	with $\sigma_1(\cdot)$ the largest singular value \cite{Zhou1998}.
	This quantity continuously depends on the elements of the matrices $A_k$, $B_k$, $C_k$ and $D_k$ at values for which the system is exponentially stable. However the function $\vec{\tau}\mapsto \|T(\cdot;\vec{\tau})\|_{\mathcal{H}_{\infty}}$ might be discontinuous, even if the system remains stable. This is caused by the potential sensitivity of the asymptotic frequency response, defined as
	\begin{equation}
		T_a(\jmath \omega;\vec{\tau}) := D_0 + \sum\limits_{k=1}^{K} D_k e^{-\jmath \omega \tau_k} \text{,}
	\end{equation}
	 with respect to infinitesimal small delay changes \cite{Gumussoy2011}.

	\begin{remark}
	 The name asymptotic frequency response stems from the following property. 
	 \end{remark}
	 	 \begin{property}[{\cite[Proposition 3.3]{Gumussoy2011}}] It holds that, \[\lim_{\bar{\omega}\to \infty}\max\left\{\sigma_1\Big(T(\jmath\omega)-T_{a}(\jmath\omega)\Big) : \omega \geq \bar{\omega}\right\} = 0\] and \[\limsup_{\omega \to \infty} \sigma_1\big(T(\jmath\omega;\vec{\tau})\big) = \sup_{\omega\in \mathbb{R}} \sigma_1\big(T_a(\jmath\omega;\vec{\tau})\big) = \|T_a(\cdot;\vec{\tau})\|_{\mathcal{H}_{\infty}}\text{.}\]
	 	 \end{property}
	 
	To eliminate this potential discontinuity with respect to the delays, one often works with the strong $\mathcal{H}_{\infty}$-norm instead:
	\[
	|||T(\cdot;\vec{\tau})|||_{\mathcal{H}_{\infty}} := \limsup_{\gamma \to 0+} \left\{\|T(\cdot;\vec{\tau}_{\gamma})\|_{\mathcal{H}_{\infty}} : \vec{\tau}_{\gamma} \in \mathcal{B}(\vec{\tau},\gamma) \cap (\mathbb{R}^{+})^{K}\right\}\text{,}
	\]
	with $\mathcal{B}(\vec{\tau},\gamma)$ a ball with radius $\gamma$ in $\mathbb{R}^{K}$ centred at $\vec{\tau}$  \cite{Gumussoy2011}. The strong $\mathcal{H}_{\infty}$-norm of the asymptotic frequency response, $|||T_a(\cdot;\vec{\tau})|||_{\mathcal{H}_{\infty}}$, is defined analogously. The strong $\mathcal{H}_{\infty}$-norm of system \eqref{eq:nominal_system_description} has the following properties:
	\begin{property}[{\cite[Theorem 4.5]{Gumussoy2011}}]
		The strong $\mathcal{H}_{\infty}$-norm is continuous as a function of the elements of the system matrices and the delays at values for which the system is exponentially stable.
	\end{property}
	\begin{property}[{\cite[Theorem 4.5]{Gumussoy2011}}] The strong $\mathcal{H}_{\infty}$-norm of system \eqref{eq:nominal_system_description} satisfies
		\label{prop:expression_strong_h_infnorm}
		\[
		|||T(\cdot;\vec{\tau})|||_{\mathcal{H}_{\infty}} = \max\big\{\|T(\cdot;\vec{\tau})\|_{\mathcal{H}_{\infty}},|||T_a(\cdot;\vec{\tau})|||_{\mathcal{H}_{\infty}}\big\}
		\]
		and
		\begin{equation}
		\label{eq:strong_H_infnorm_asymptotic}
		|||T_a(\cdot;\vec{\tau})|||_{\mathcal{H}_{\infty}} = \max_{\vec{\theta} \in \left[0,2\pi\right)^{K}} \sigma_1\left(D_0 + \sum_{k=1}^{K} D_k e^{\jmath \theta_k}\right)\text{.}
		\end{equation} 
	\end{property}
	
	\begin{property}[{\cite[Proposition 4.3]{Gumussoy2011}}]
		\label{prop:strong_hinf_norm_rationally_independent}
		It holds that,
		\[
		|||T_a(\cdot;\vec{\tau})|||_{\mathcal{H}_{\infty}} \geq \|T_a(\cdot;\vec{\tau})\|_{\mathcal{H}_{\infty}} \text{.}
		\]
		Furthermore, if the delays are rationally independent, then 
		\[
		|||T_a(\cdot;\vec{\tau})|||_{\mathcal{H}_{\infty}} = \|T_a(\cdot;\vec{\tau})\|_{\mathcal{H}_{\infty}}\quad \text{and as a consequence} \quad |||T(\cdot;\vec{\tau})|||_{\mathcal{H}_{\infty}} = \|T(\cdot;\vec{\tau})\|_{\mathcal{H}_{\infty}}\text{.}
		\]
	\end{property}
	The following example illustrates the potential sensitivity of the $\mathcal{H}_{\infty}$-norm with respect to infinitesimal delay changes in greater detail.
	\begin{example} 
		\label{example:strong_hinfnorm}
		Let us consider the following system
		\begin{equation}
		\label{eq:example0_system_description}
		\left\{ 
		\begin{array}{rcllllll}
		\dot{x}(t) & = & -2x(t) & +x(t-\tau_1) &-w(t) & - 0.5 w(t-\tau_2)&\\ [2ex]
		z(t) & = & -2 x(t) & +  x(t-\tau_2)&+ 5w(t) &+ 1.5 w(t-\tau_1)&-3w(t-\tau_2)
		\end{array}
		\right.
		\end{equation}
		whose corresponding transfer function and asymptotic frequency response function are respectively equal to 
		\[
		\begin{aligned}
		T(s;(\tau_1,\tau_2)) &= \frac{(-2+e^{-s\tau_2})(-1-0.5e^{-s\tau_2})}{s+2-e^{-s\tau_1}} + 5 + 1.5 e^{-s\tau_1}-3e^{-s\tau_2} \text{~and} \\
		T_a(\jmath \omega;(\tau_1,\tau_2)) &= 5 + 1.5 e^{-\jmath\omega\tau_1}-3e^{-\jmath\omega\tau_2} \text{.}
		\end{aligned}
		\]
		Figure \ref{fig:T_strong_hinfnorm} plots the magnitude of the frequency response and the asymptotic frequency response in function of $\omega$ for $\tau_1$ and $\tau_2$ respectively equal to $1$ and $2$. For these values of delays, the magnitude of the frequency response attains a maximum of approximately $8.4907$. Figure \ref{fig:T_strong_hinfnorm_2} shows the effect of a small change in delay parameters ($\tau_1 =1$ and $\tau_2= 2+\pi/100$). Now the $\mathcal{H}_{\infty}$-norm is equal to approximately $9.5$. Furthermore, from Properties~\ref{prop:expression_strong_h_infnorm} and \ref{prop:strong_hinf_norm_rationally_independent} it follows that 
		\[
	\begin{aligned}		
		\|T((1,2+\pi/n))\|_{\mathcal{H}_{\infty}} &= |||T((1,2+\pi/n))|||_{\mathcal{H}_{\infty}} \\ & \geq |||T_a((1,2+\pi/n))|||_{\mathcal{H}_{\infty}} \\ &= \max_{\vec{\theta}\in[0,2\pi)^{2}} \sigma_1\left(5+1.5e^{\jmath\theta_1}-3e^{\jmath\theta_2} \right)\\ &= 9.5
	\end{aligned}
	\] 
		for every $n\in \mathbb{N}$. One can thus choose delays arbitrarily close to $(1,2)$ for which the $\mathcal{H}_{\infty}$-norm jumps to at least $9.5$. From the figures it is also clear that this discontinuity is due to the asymptotic frequency response.
		\begin{figure}[!htbp]
			\centering
			\begin{subfigure}{0.47\linewidth}
				\centering
				\includegraphics[width=\linewidth]{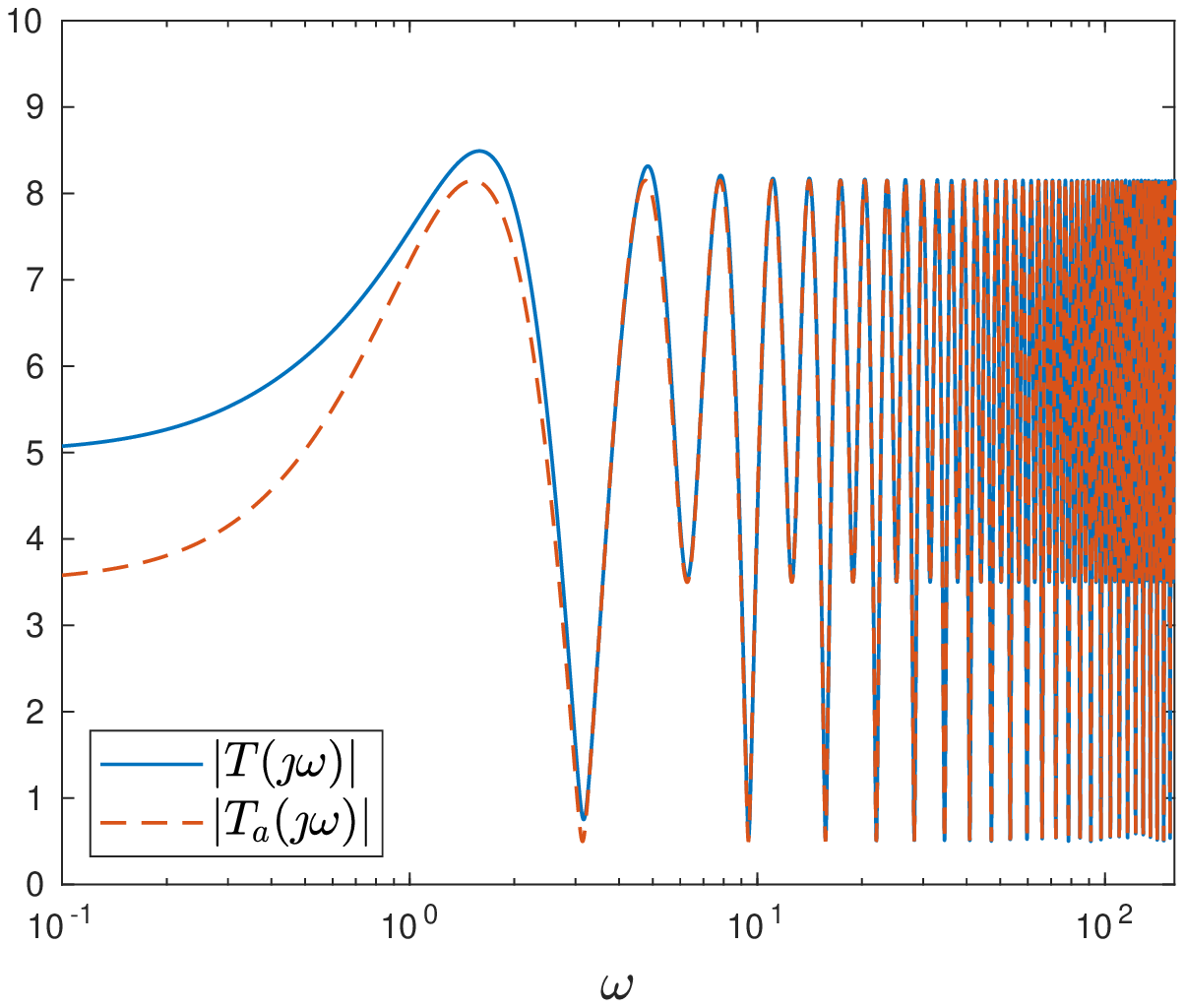}
				\caption{}
				\label{fig:T_strong_hinfnorm}
			\end{subfigure}
			\begin{subfigure}{0.47\linewidth}
				\centering
				\includegraphics[width=\linewidth]{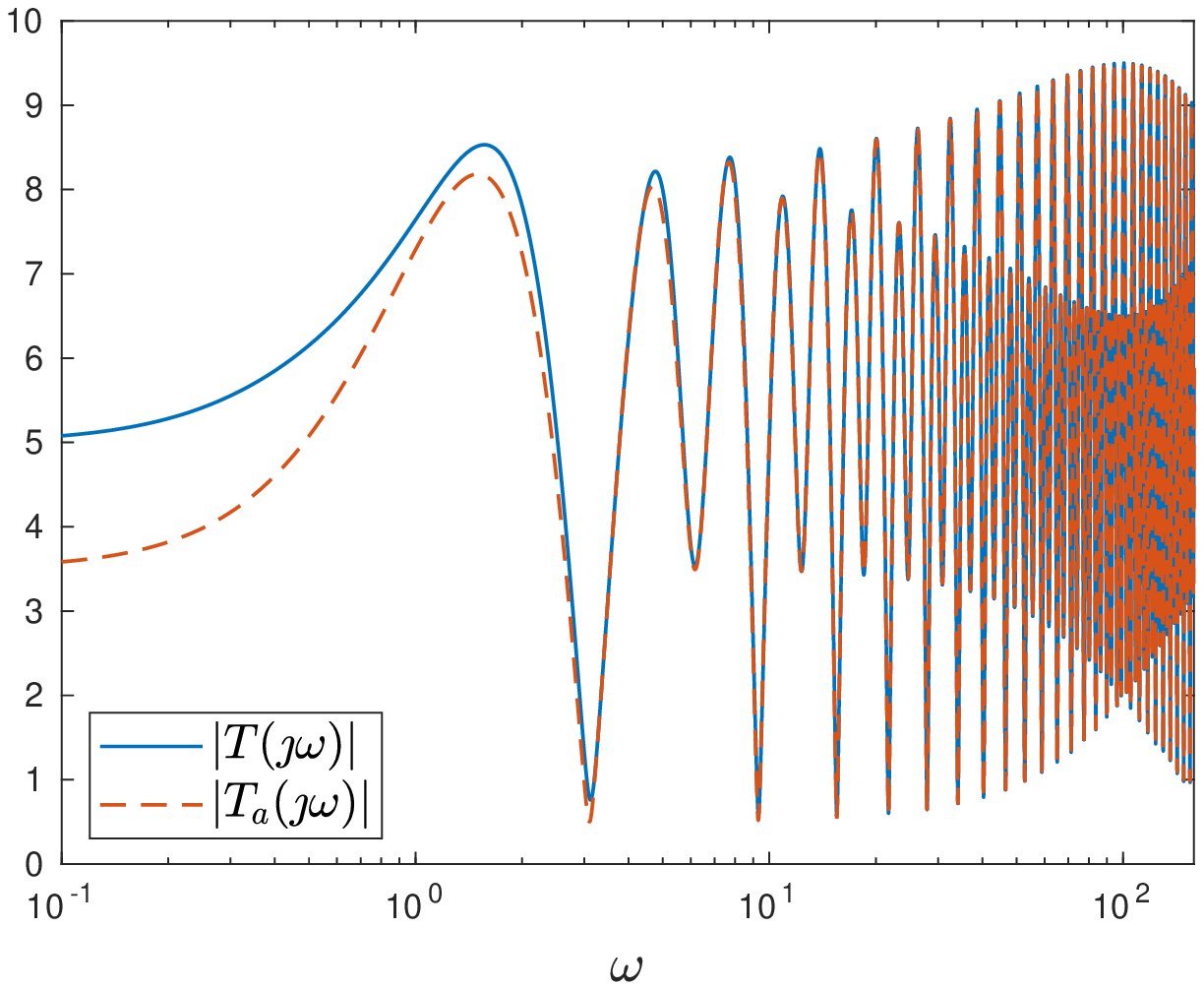}
				\caption{}
				\label{fig:T_strong_hinfnorm_2}
			\end{subfigure}
			\caption{The magnitude of the frequency response (blue) and asymptotic frequency response (red) of system \eqref{eq:example0_system_description} in function of $\omega$ for $(\tau_1,\tau_2)$ equal to $(1,2)$ (left) and $(1,2+\pi/100)$ (right).}
		\end{figure}

	\end{example}
	
	By considering the strong $\mathcal{H}_{\infty}$-norm we remove a \emph{fragility} problem of the {$\mathcal{H}_{\infty}$-norm}, namely being potentially sensitive to infinitesimal perturbations on the delays. Note that the strong \mbox{$\mathcal{H}_{\infty}$-norm} is still a property of the nominal model (and infinitesimal perturbations of the delays). However, in almost all control design applications the mathematical model does not completely match the dynamical system it describes, due to unmodelled (non-linear) behaviour, model reductions, imprecise measurements or uncertain parameters. To take these deviations into account during the design process one often works with a family of models instead  \cite{Zhou1998,Hinrichsen2005}. In this paper, we construct such a family by considering \eqref{eq:nominal_system_description} as nominal model to which uncertainties are added. In the main part of this work we will consider real-valued (as the model of the system itself is real-valued), norm-bounded (reflecting the distance between model and reality) and structured (only a certain parameter or group of parameters is affected) uncertainties on the coefficient matrices. However in Section~\ref{sec:generalisation} also uncertainties on the delays will be examined. The state-space representation associated with the considered uncertain system is equal to:
	\begin{equation}
	\label{eq:system_description}
	\left\{\begin{array}{lr}
	\dot{x}(t)=&\tilde{A}_0(\delta)~x(t) + \sum\limits_{k=1}^{K}\tilde{A}_k(\delta)~x(t-\tau_k) + \tilde{B}_0(\delta)~w(t) + \sum\limits_{k=1}^{K} \tilde{B}_{k}(\delta)~w(t-\tau_k)  \\ [1em]
	z(t)=&\!\!\tilde{C}_0(\delta)~x(t) + \sum\limits_{k=1}^{K}\tilde{C}_k(\delta)~x(t-\tau_k) + \tilde{D}_0(\delta)~w(t) + \sum\limits_{k=1}^{K} \tilde{D}_{k}(\delta)~w(t-\tau_k)
	\end{array}\right.
	\end{equation} 
	where the uncertainties $\delta$ are confined to a specified set $\hat{\delta}$. In this formulation $x \in \mathbb{R}^{n}$ is the state vector, $w\in\mathbb{R}^{m}$ the exogenous input, $z\in\mathbb{R}^{p}$ the exogenous output, $\vec{\tau} = (\tau_1,\tau_2,\dots,\tau_K) \in (\mathbb{R}^{+})^{K}$ discrete delays, $\delta$ the combination of all uncertainties: $\delta = \left(\delta_1,\dots,\delta_L\right)$, $\hat{\delta}$ the set of admissible uncertainties: \[\hat{\delta} = \{\delta \in \mathbb{R}^{q_1\times r_1}\times\dots\times \mathbb{R}^{q_L\times r_L}: \|\delta_l\|_F\leq \bar{\delta}_l \text{ for } l=1,\dots,L\}\text{,}\] $\tilde{A}_k(\delta)$, $\tilde{B}_k(\delta)$, $\tilde{C}_k(\delta)$, and $\tilde{D}_k(\delta)$ uncertain system matrices of appropriate dimension with $\tilde{A}_k(\delta) = A_k + \sum_{l=1}^{L}\sum_{s=1}^{S_l^{A_k}} G_{l,s}^{A_k} \delta_l H_{l,s}^{A_k}$ where $G_{l,s}^{A_k}$ and $H_{l,s}^{A_k}$ are real-valued shape matrices of appropriate dimension and $\tilde{B}_k(\delta)$, $\tilde{C}_k(\delta)$ and $\tilde{D}_k(\delta)$ defined analogously. Note that this definition allows a single uncertainty to affect multiple blocks in the same system matrix and even multiple system matrices as the same uncertain parameter may be present at multiple locations. In the remainder of this work we make the following assumption for this uncertain system:
	\begin{assumption}
		\label{asm:robust_internally_stable}
		 System \eqref{eq:system_description} is internally exponentially stable for all admissible uncertainties, ie. the characteristic roots of $\lambda I- \tilde{A}_0(\delta) -\tilde{A}_k(\delta) e^{-\lambda \tau_k}$ lie in the open left half plane for all $\delta \in \hat{\delta}$. 
	\end{assumption}
	
	For such uncertain systems it is often desirable to quantify the worst behaviour over all possible realisations. In the context of the $\mathcal{H}_{\infty}$-norm this led to the notion of the robust $\mathcal{H}_{\infty}$-norm which is defined as the maximal $\mathcal{H}_{\infty}$-norm over all realisations, ie.
	\[	
		\|T(\cdot;\cdot,\vec{\tau})\|^{\hat{\delta}}_{\mathcal{H}_{\infty}} := \max_{\delta\in\hat{\delta}} \| T(\cdot;\delta,\vec{\tau})\|_{\mathcal{H}_{\infty}} 
	\]
	where $T(s;\delta,\vec{\tau})$ is the transfer functions associated with a given realisation of \eqref{eq:system_description}:	
	\begin{equation}
	\label{eq:transferfunction}
	\begin{array}{lc}
	T(s;\delta,\vec{\tau}) = & \!\left(\!\tilde{C}_0(\delta)\plus\sum\limits_{k=1}^{K}\tilde{C}_k(\delta) e^{\minus\tau_k s}\right)\!\left(\!Is\minus\tilde{A}_0(\delta)\minus\sum\limits_{k=1}^{K}\tilde{A}_k(\delta) e^{\minus\tau_k s}\!\right)^{\minus 1}\!\!\left(\!\tilde{B}_0(\delta) \plus\sum\limits_{k=1}^{K}\tilde{B}_k(\delta) e^{\minus \tau_k s}\!\right) \\ [0.7em] &+ \tilde{D}_0(\delta) \plus \sum\limits_{k=1}^{K} \tilde{D}_k e^{\minus\tau_k s}\text{.}
	\end{array} 
	\end{equation}
	The robust $\mathcal{H}_{\infty}$-norm can also be interpreted as the supremum of the following (worst-case gain) function
	\begin{equation}
	\label{eq:worst_case_transfer_function}
	\mathbb{R} \ni \omega  \mapsto \max_{\delta \in  \hat{\delta}} \sigma_{1}\Big(T(\jmath \omega; \delta,\vec{\tau})\Big) \in \mathbb{R}^{+} \text{, }
	\end{equation} which for each frequency gives the maximal (over both the input signals and the admissible realisations) input-output gain of the system.
	
	 However, the potential discontinuity of the nominal $\mathcal{H}_{\infty}$-norm with respect to the delays caries over to the robust $\mathcal{H}_{\infty}$-norm. Therefore this paper works with the robust strong $\mathcal{H}_{\infty}$-norm instead:
	\begin{align}
	\label{eq:robust_strong_h_inf_norm}
	|||T(\cdot;\cdot,\vec{\tau})|||^{\hat{\delta}}_{\mathcal{H}_{\infty}} &:= \max_{\delta\in\hat{\delta}} ||| T(\cdot;\delta,\vec{\tau})|||_{\mathcal{H}_{\infty}}\notag 
	\intertext{which, by Property~\ref{prop:expression_strong_h_infnorm}, is equal to}
	|||T(\cdot;\cdot,\vec{\tau})|||^{\hat{\delta}}_{\mathcal{H}_{\infty}} &= \max\Big\{\underbrace{\max_{\delta\in\hat{\delta}}\|T(\cdot;\delta,\vec{\tau})\|_{\mathcal{H}_{\infty}}}_{\|T(\cdot;\cdot,\vec{\tau})\|_{\mathcal{H}_{\infty}}^{\hat{\delta}}},\underbrace{\max_{\delta\in\hat{\delta}}|||T_a(\cdot;\delta,\vec{\tau})|||_{\mathcal{H}_{\infty}}}_{|||T_a(\cdot;\cdot,\vec{\tau})|||_{\mathcal{H}_{\infty}}^{\hat{\delta}}}\Big\}  \text{.}
	\end{align}
	From this definition it follows that either $ |||T(\cdot;\cdot,\vec{\tau})|||^{\hat{\delta}}_{\mathcal{H}_{\infty}} = |||T_a(\cdot;\cdot,\vec{\tau})|||_{\mathcal{H}_{\infty}}^{\hat{\delta}}$ or  $|||T(\cdot;\cdot,\vec{\tau})|||^{\hat{\delta}}_{\mathcal{H}_{\infty}} =\|T(\cdot;\cdot,\vec{\tau})\|_{\mathcal{H}_{\infty}}^{\hat{\delta}}>|||T_a(\cdot;\cdot,\vec{\tau})|||_{\mathcal{H}_{\infty}}^{\hat{\delta}} $. In the former case the robust strong $\mathcal{H}_{\infty}$-norm is equal to the worst-case value of the strong $\mathcal{H}_{\infty}$-norm of the asymptotic transfer function (which we will call the robust strong asymptotic $\mathcal{H}_{\infty}$-norm in the remainder of this paper). In the latter case the robust strong $\mathcal{H}_{\infty}$-norm is equal to maximum of the worst-case gain function (which is attained at finite frequencies). We comeback to this characterisation in Section~\ref{subsec:central_theorem}.

	The existing numerical methods to compute the nominal $\mathcal{H}_{\infty}$-norm (of delay free systems) can be divided in two classes. A first group \cite{Benner2018,Boyd1989,Gumussoy2011} is based on the BBBS level-set algorithm presented in \cite{Boyd1989}. These methods repeatedly compute the spectrum of an associated Hamiltonian eigenvalue problem and check for strictly imaginary (ie. real part equal to zero) eigenvalues. Because the cost of this last operation increases cubically with the size of the state matrix, this method is rather slow for large systems. The second class \cite{Benner2014,Guglielmi2013a} avoids this complete eigenvalue decomposition by using the relation between the $\mathcal{H}_{\infty}$-norm and the structured distance to instability (also known as the stability radius \cite{Hinrichsen2005}) of an associated singular eigenvalue problem with a structured, complex-valued perturbation. More specifically:
	\begin{proposition}[{\cite[Proposition 3.2]{Benner2014}}]
		\label{prop:dist_inst_hinf}
		The $\mathcal{H}_{\infty}$-norm of 
		\[T(s) = C(Is-A)^{-1}B + D\] 
		is equal to the reciprocal of the structured distance to instability of
		\begin{equation}
		\label{eq:associated_EP}
		M(\lambda;\Delta) := 
		\begin{bmatrix}
		I & 0 & 0 \\
		0 & 0 & 0 \\
		0 & 0 & 0
		\end{bmatrix}\lambda - \left(\begin{bmatrix}
		A & B & 0 \\
		0 & -I& 0 \\
		C & D & -I
		\end{bmatrix} + \begin{bmatrix}
		0 \\ I \\ 0
		\end{bmatrix} \Delta \begin{bmatrix}
		0 & 0 & I
		\end{bmatrix}\right) \text{,}
		\end{equation} which is defined as the smallest $\epsilon$ such that there exists a $\Delta \in \mathbb{C}^{m\times p}$ with $\| \Delta \|_2 \leq \epsilon$ for characteristic matrix \eqref{eq:associated_EP} is not well-posed (see later on) or has (a) characteristic root(s) in the closed right half plane.
	\end{proposition}
\noindent  The main computation cost of these last algorithms stems from calculating the right-most eigenvalues of \eqref{eq:associated_EP} for several $\Delta$. For large, sparse matrices these right-most eigenvalues can efficiently be computed using specialised iterative methods such as \cite{Meerbergen1994,Lehoucq1998}. The method presented in this paper fits in this last framework.

	The remainder of this paper is organised as follows. In Section~\ref{sec:preliminaries} we revise some theory related to (perturbed) singular delay eigenvalue problems. Section~\ref{sec:rob_strong_hinf_dist_instab} contains the main theoretical result of the paper as it gives the relation between the robust strong $\mathcal{H}_{\infty}$-norm of system \eqref{eq:system_description} and the (robust structured complex) distance to instability of an associated singular delay eigenvalue problem. Next, Section~\ref{sec:numerical_algorithm} presents a numerical algorithm, based on this relation, to compute the robust strong $\mathcal{H}_{\infty}$-norm. Subsequently Section~\ref{sec:generalisation} generalises the theory and the presented method to systems with uncertainties on both the coefficient matrices and the delays, and to systems whose nominal model is represented by delay differential algebraic equations. Finally, Sections \ref{sec:examples} and \ref{sec:conclusions} give some numerical examples and concluding remarks. 
	\section{Singular delay eigenvalue problems}
	\label{sec:preliminaries}
	As mentioned in Proposition~\ref{prop:dist_inst_hinf}, there exists a link between the $\mathcal{H}_{\infty}$-norm of a delay free system and the structured distance to instability of an associated singular eigenvalue problem. Section~\ref{sec:rob_strong_hinf_dist_instab} introduces a similar relation between the robust strong $\mathcal{H}_{\infty}$-norm of system \eqref{eq:system_description} and the (robust structured complex) distance to instability of a singular delay eigenvalue problem (SDEP) whose nominal characteristic matrix has the following structure
	\begin{equation}
	\label{eq:model_SDEP}
	M(\lambda;\vec{\tau}) := Q\lambda - P_0 -\sum_{k=1}^{K} P_k e^{-\lambda \tau_k} \
	\end{equation}
	with $Q$, $P_{0}$, ..., $P_{K}$ square matrices of dimension $n_{M}~(=n+m+p)$ and $\vec{\tau} = (\tau_1,\tau_2,\dots,\tau_K) \in (\mathbb{R}^{+})^{K}$ discrete delays. The matrix $Q$ can be singular and in the remainder of this paper $U_{\mathcal{N}}$ and $V_{\mathcal{N}}$ will denote $n_{M}\times (n_{M}-\rank(Q))$-dimensional matrices whose columns form a basis for respectively the left and right nullspace of $Q$.
	The behaviour of such eigenvalue problems is however non-trivial. Therefore this section will revise some related theory. First the focus lies on the  nominal eigenvalue problem (Section~\ref{subsec:SDEP}). Subsequently the effect of structured perturbations is examined (Section~\ref{subsec:distane_to_instability}).

	\subsection{Spectral properties}
	\label{subsec:SDEP}
	To get a better understanding of the eigenvalue problem associated with \eqref{eq:model_SDEP}, we first examine some properties of a singular eigenvalue problem without delays in the characteristic matrix:
	\begin{equation}
		\label{eq:matrix_pencil}
		N(\lambda) := \; Q\lambda-P_0 \text{.}
	\end{equation}
	This characteristic matrix is called regular when its characteristic polynomial \big($\mathbb{C} \ni \lambda\mapsto\det(N(\lambda))$\big) does not vanish identically \cite{Fridman2002}. If \eqref{eq:matrix_pencil} is regular, it can be transformed to Weierstrass-Kronecker canonical form \cite{Fridman2002,Gantmacher1959}: there exist nonsingular matrices $W$ and $T$ such that
	\begin{equation*}
		WQT = \begin{bmatrix}
		I_{\mu} & 0 \\
		0 & N
		\end{bmatrix}\text{ and }
		WP_{0}T = 
		\begin{bmatrix}
		J & 0 \\
		0 & I_{n-\mu}
		\end{bmatrix}
	\end{equation*}
	with $\mu$ the sum of the algebraic multiplicities of all finite eigenvalues, $J$ in Jordan form and $N$ a nilpotent matrix in Jordan form. The index of \eqref{eq:matrix_pencil} is defined as the smallest integer $\nu$ such that $N^{\nu}=0$ (and the index is equal to $0$ if $N$ is void).  
	
	 We now return to our original SDEP. As in \cite{Fridman2002}, characteristic matrix \eqref{eq:model_SDEP} is said to be regular if $Q\lambda-P_0$ is regular and (if regular) its index is equal to the index of $Q\lambda - P_0$. Based on these notions of regularity and index, we present well-posedness in the remainder of this work in the following way:
	\begin{definition}
		Characteristic matrix \eqref{eq:model_SDEP} is called well-posed when it is regular and has at most index 1.
	\end{definition}
\noindent The following lemma allows us to easily verify this well-posedness condition.
	\begin{lemma}[{\cite[Lemma 2]{Bunse-Gerstner1999}}]
		\label{lem:condition_wellposed}
		Characteristic matrix \eqref{eq:model_SDEP} is well-posed if and only if ${U_{\mathcal{N}}}^{H}P_0V_{\mathcal{N}}$ is non-singular. 
	\end{lemma}

	Next we restrict ourself to well-posed characteristic matrices and examine some properties of their (finite) spectrum:
	\begin{equation}
	\label{eq:spectrum_SDEP}
	\Lambda(\vec{\tau}) := \Big\{\lambda \in \mathbb{C} : \det(M(\lambda;\vec{\tau})) = 0\Big\} \text{.}
	\end{equation}
	Because neutral delay eigenvalue problems can be reformulated in form \eqref{eq:model_SDEP} (see \cite{Michiels2011}), some properties of neutral delay eigenvalue problems carry over to the studied SDEPs. More specifically the spectral abscissa of \eqref{eq:spectrum_SDEP}, ie.  
 	\[
 		\alpha(\vec{\tau}) := \sup\left\{\Re(\lambda) : \lambda \in \Lambda(\vec{\tau}) \right\} \text{,}
 	\]
 	may be discontinuous with respect to the delays. We therefore consider the strong spectral abscissa \cite{Michiels2011}: 
	\begin{equation*}
		\alpha_{s}(\vec{\tau}) :=\limsup_{\gamma\to 0+}\{\alpha(\vec{\tau}_\gamma): \vec{\tau}_{\gamma} \in \mathcal{B}(\vec{\tau},\gamma)\cap (\mathbb{R}^{+})^{K}\}\text{,}
	\end{equation*}
	which has the following property.

 \begin{property}[{\cite[Propostition~3]{Michiels2011}}]
 	\label{prop:definition_strong_spectral_abscissa}
	The strong spectral abscissa satisfies
	\[
	\alpha_{s}(\vec{\tau}) = \max\Big\{\alpha_{D,s}(\vec{\tau}),\alpha(\vec{\tau}) \Big\}  
	\]
	with 
  	$\alpha_{D,s}(\vec{\tau})$ equal to the zero crossing of
	\begin{equation}
	\label{eq:spectral_abcissa_1} 
	\mathbb{R} \ni \varsigma  \mapsto \max_{\vec{\theta} \in \left[0,2\pi\right)^K} \rho\left(\sum_{k=1}^{K}\left({U_{\mathcal{N}}}^{H}P_{0}V_{\mathcal{N}}\right)^{-1}\left({U_{\mathcal{N}}}^{H}P_{k}V_{\mathcal{N}}\right)e^{-\varsigma\tau_k+\jmath\theta_k}\right) -1\text{,}
	\end{equation}
	where $\rho(\cdot)$ is the spectral radius, if such a crossing exists and otherwise $\alpha_{D,s}(\vec{\tau}) = -\infty$. In addition, the strong spectral abscissa is continuous with respect to both the elements of $P_0,\dots,P_K$ and the delays $\vec{\tau}$ as long as ${U_{\mathcal{N}}}^{H}P_{0}V_{\mathcal{N}}$ remains non-singular.
\end{property}


 Finally, we introduce the following definition of strong stability based on the notions of well-posedness and strong spectral abscissa.
 \begin{definition}
 	\label{def:strongly_stable}
	A characteristic matrix $M(\lambda;\vec{\tau})$ is strongly stable if it is well-posed and its strong spectral abscissa is strictly negative.
 \end{definition}

	\subsection{Robust structured complex distance to instability} 
	\label{subsec:distane_to_instability}
	This subsection examines the effect of adding perturbations to a strongly stable characteristic matrix. Inspired by Proposition~\ref{prop:dist_inst_hinf}, we focus on a characteristic matrix of the following form: 
	\small{
	\begin{equation}
	\label{eq:SDEP}
	M(\lambda;\delta,\Delta,\vec{\tau}) := 
	\begingroup
	\setlength\arraycolsep{3pt}
	\underbrace{\begin{bmatrix}
		I_n & 0 & 0\\
		0   & 0 & 0\\
		0   & 0 & 0
		\end{bmatrix}}_{Q}
	\lambda\minus
	\underbrace{\begin{bmatrix}
		\tilde{A}_0(\delta) & \tilde{B}_0(\delta) & 0\\
		0   & -I_m& 0 \\
		\tilde{C}_0(\delta) & \tilde{D}_0(\delta) &  -I_p
		\end{bmatrix} \minus
		\begin{bmatrix}
			0 \\
			I_m \\
			0 
			\end{bmatrix}
		\Delta
	\begin{bmatrix} 
			0 & 0 & I_p
		\end{bmatrix}}_{\tilde{P}_0(\delta,\Delta)}
	\!\minus \sum_{k=1}^{K}
	\underbrace{\begin{bmatrix}
		\tilde{A}_k(\delta) & \tilde{B}_k(\delta) & 0 \\
		0   &  0  & 0 \\
		\tilde{C}_k(\delta) & \tilde{D}_k(\delta) & 0
		\end{bmatrix}}_{\tilde{P}_k(\delta)} e^{\minus \lambda \tau_k}
	\endgroup \text{,}
	\end{equation}}\normalsize
with $\tilde{A}_k$, $\tilde{B}_k$, $\tilde{C}_k$, $\tilde{D}_k$ and $\vec{\tau}$ as defined in Section~\ref{sec:introduction} and $\Delta \in \mathbb{C}^{m \times p}$. Observe that characteristic matrix \eqref{eq:SDEP} shares some similarities with \eqref{eq:associated_EP}. But there are two main differences: firstly \eqref{eq:SDEP} contains delay terms (due to the discrete delays in \eqref{eq:system_description}) and secondly \eqref{eq:SDEP} has both real- and complex-valued perturbations. The real-valued perturbations ($\delta$) originate from the uncertainties in model \eqref{eq:system_description} and are therefore confined to the set $\hat{\delta}$. The complex valued perturbation ($\Delta$) on the other hand plays a similar role as in Proposition \ref{prop:dist_inst_hinf}: we are interested in the smallest $\epsilon$ such that there exists a $\Delta \in \mathbb{C}^{m\times p}$ with $\|\Delta\|_2 \leq \epsilon$ for which $M(\lambda;\delta,\Delta,\vec{\tau})$ is not strongly stable for at least one $\delta \in \hat{\delta}$. This critical $\epsilon$ will be called the \emph{robust} (worst-case value over all permissible real-valued perturbations) structured \emph{complex} (to emphasise that only the bound on the complex-valued perturbation is varied) distance to instability. In Section~\ref{sec:rob_strong_hinf_dist_instab} it will be shown that there exists a relation between this robust structured complex distance to instability and the robust strong \mbox{$\mathcal{H}_{\infty}$-norm} of system \eqref{eq:system_description}, while in the remainder of this subsection we characterise this robust structured complex distance to instability in greater detail.

	 From the definition of strong stability (see Definition \ref{def:strongly_stable}), it follows that there are two ways in which a loss of strong stability can occur. Firstly, the characteristic matrix can become non well-posed. The corresponding robust structured complex distance to non well-posedness is defined as:
	\begin{equation}
		\dist_{NWP}(\hat{\delta})  = 
		\begin{cases}
		+\infty, \text{ if $M(\lambda;\delta,\Delta,\vec{\tau})$ is well-posed for all $\Delta \in \mathbb{C}^{m\times p}$ and $\delta\in\hat{\delta}$} \\
		\min_{\substack{\delta \in \hat{\delta} \\ \Delta \in \mathbb{C}^{m \times p}}}\{\|\Delta\|_2 : \text{$M(\lambda;\delta,\Delta,\vec{\tau})$ is not well-posed} \}\text{,  otherwise}
		\end{cases} \text{.}
	\end{equation}

	Secondly, a realisation of \eqref{eq:SDEP} can loose strong stability if its strong spectral abscissa becomes non-negative. Therefore we study the $(\hat{\delta},\epsilon)$-strong pseudo-spectral abscissa of \eqref{eq:SDEP}, which for $\epsilon \in \left[0,\dist_{NWP}\right)$ is defined as the maximal strong spectral abscissa over all realisations of \eqref{eq:SDEP} with $\delta \in \hat{\delta}$ and $\|\Delta\|_2 \leq \epsilon$:
	\[
	 	\alpha_{s}^{\mathrm{ps}}(\hat{\delta},\epsilon,\vec{\tau}) :
 	 	=\max_{\delta \in \hat{\delta}}\max_{\substack{\Delta \in \mathbb{C}^{m\times p}\\ \|\Delta\|_2 \leq \epsilon}} \alpha_{s}(\delta,\Delta,\vec{\tau})
 	 	\]
		with $\alpha_{s}(\delta,\Delta,\vec{\tau})$ the strong spectral abscissa of $M(\lambda;\delta,\Delta,\tau)$. Using Property~\ref{prop:definition_strong_spectral_abscissa} this leads to
		\begin{equation}
		\label{eq:strong_psa}
	 	\alpha_{s}^{\mathrm{ps}}(\hat{\delta},\epsilon,\vec{\tau}) 
	 	= \max\bigg\{\underbrace{\max_{\delta \in  \hat{\delta}}\max_{\substack{\Delta \in \mathbb{C}^{m\times p}\\ \|\Delta\|_2 \leq \epsilon}}\alpha_{D,s}(\delta,\Delta,\vec{\tau})}_{\alpha_{D,s}^{\mathrm{ps}}(\hat{\delta},\epsilon,\vec{\tau})},\underbrace{\max_{\delta \in  \hat{\delta}}\max_{\substack{\Delta \in \mathbb{C}^{m\times p}\\ \|\Delta\|_2 \leq \epsilon}}\alpha(\delta,\Delta,\vec{\tau})}_{\alpha^{\mathrm{ps}}(\hat{\delta},\epsilon,\vec{\tau})}\bigg\} \text{.} 	 	
		\end{equation}
	\begin{remark}
		\label{remark:pseudo_spectral_abscissa}
		The value $\alpha^{\mathrm{ps}}(\hat{\delta},\epsilon,\vec{\tau})$ can be interpreted as the supremum of the real part of the points in the $(\hat{\delta},\epsilon)$-pseudo-spectrum of \eqref{eq:SDEP}, ie.
		\[
		\alpha^{\mathrm{ps}}(\hat{\delta},\epsilon,\vec{\tau}) = \sup\{\Re(\lambda) : \lambda \in \Lambda^{\mathrm{ps}}(\hat{\delta},\epsilon,\vec{\tau})\}
		\]
		with
		\begin{equation}
		\label{eq:combined_real_complex_pseudospectrum}
		\Lambda^{\mathrm{ps}}(\hat{\delta},\epsilon,\vec{\tau}) = \bigcup_{\delta \in \hat{\delta}}\bigcup_{\substack{\Delta \in \mathbb{C}^{m\times p} \\ \|\Delta\|_2\leq\epsilon}} \left\{\lambda \in \mathbb{C} : \det\big(M(\lambda;\delta,\Delta,\vec{\tau}) \big) = 0 \right\} \text{.}
		\end{equation}
	\end{remark}
	 Based on \eqref{eq:strong_psa} we now introduce two other distance measures. Firstly, the robust structured complex distance to a characteristic root chain crossing is defined as
	\begin{equation}
	\dist_{CHAIN}(\hat{\delta}) = 
	\begin{cases}
	+\infty\text{, if } \alpha_{D,s}^{\mathrm{ps}}(\hat{\delta}, \epsilon,\vec{\tau})< 0 \quad \forall \epsilon \in \left[0, \dist_{NWP}\right) \\
	\min\{\varepsilon \in \left[0, \dist_{NWP}\right): \alpha_{D,s}^{\mathrm{ps}}(\hat{\delta},\epsilon,\vec{\tau}) \geq 0 \} \text{, }  \text{otherwise} 
	\end{cases} \text{.}
	\end{equation}
	Secondly, the robust structured complex distance to finite root crossing is defined as
	\begin{equation}
	\label{eq:definition_dist_fin}
	\dist_{FIN}(\hat{\delta})  =
	\begin{cases}
	+\infty\text{, if } \alpha^{\mathrm{ps}}(\hat{\delta},\epsilon,\vec{\tau})< 0 \quad \forall \epsilon \in \left[0,\min\{\dist_{NWP},\dist_{CHAIN}\}\right)\\
	\min\{\epsilon \in \left[0,\min\{\dist_{NWP},\dist_{CHAIN}\}\right) :  \alpha^{\mathrm{ps}}(\hat{\delta},\epsilon,\vec{\tau}) \geq 0 \}\text{, }  \text{otherwise} 
	\end{cases} \text{.}
	\end{equation}
	The names of these two distances can be understood using the following property.
	\begin{property}\cite{Michiels2007}
		\label{prop:strong_difference_spectral_abscissa}
		If $\alpha_{D,s}(\vec{\tau}) \geq 0$ then for each $\gamma>0$ there exist $\vec{\tau}_{\gamma} \in \mathcal{B}(\vec{\tau},\gamma)\cap (\mathbb{R}^{+})^{K}$ and $c\geq0$ such that $\Lambda(\vec{\tau}_{\gamma})$ contains a chain of characteristic roots $\{\lambda_i \}_{i\in\mathbb{N}}$  that satisfies
\[\lim\limits_{i\to \infty} \Re(\lambda_i) = c \qquad \lim\limits_{i\to \infty} |\Im(\lambda_i)| = +\infty \text{.}\] 
		If $\alpha_{D,s}(\vec{\tau})<0$, there exist $\epsilon_1 > 0$ and $\epsilon_2 > 0 $ such that for any $\vec{\tau}_{\epsilon_1} \in \mathcal{B}(\vec{\tau},\epsilon_1)\cap (\mathbb{R}^{+})^{K}$ the number of characteristic roots of $M(\lambda;\vec{\tau}_{\epsilon_1})$ that lie to the right of $-\epsilon_2$ is finite. 
	\end{property}
	A finite robust structured complex distance to a characteristic root chain crossing is thus the smallest $\epsilon$ for which there exist $\delta\in \hat{\delta}$ and $\Delta \in \mathbb{C}^{m\times p}$ with $\|\Delta\|_2 \leq \epsilon$ such that the spectrum of the associated realisation of \eqref{eq:SDEP} contains a chain of characteristic roots with a vertical asymptote in the closed right-half plane, for some delays $\vec{\tau}_\gamma$ that can be chosen arbitrarily close to $\vec{\tau}$. At the same time a finite robust structured complex distance to finite root crossing corresponds to the smallest $\epsilon$ for which there exists $\delta\in \hat{\delta}$ and $\Delta \in \mathbb{C}^{m\times p}$ with $\|\Delta\|_2 \leq \epsilon$ such that the spectrum of $M(\lambda;\delta,\Delta,\vec{\tau})$ has (finitely many) eigenvalues in the closed right half plane (note that for $\epsilon< \dist_{CHAIN}(\hat{\delta})$ the number of roots in the closed right-half plane is finite, even for infinitesimal small delay perturbations). These concepts are illustrated in Examples \ref{example:finite_crossing}, \ref{example:chain_crossing} and \ref{example:strong_chain_crossing} in Section \ref{sec:rob_strong_hinf_dist_instab}.

	\begin{remark}
		\label{remark:critical_epsilon}
		Because the strong spectral abscissa of a given realisation of \eqref{eq:SDEP} is continuous with respect to the elements of $\Delta$ as long as $\|\Delta\|_2<\dist_{NWP}$ (a consequence of Property~\ref{prop:definition_strong_spectral_abscissa}), a transition to a non-negative pseudo-spectral abscissa is characterised by a critical $\epsilon^{\star}$ for which $\alpha_{s}^{\mathrm{ps}}(\hat{\delta},\epsilon^{\star},\vec{\tau}) = 0$.
	\end{remark}
	
	The robust structured complex distance to instability now can be expressed in function of the three distance measures defined above:
	\begin{equation}
	\dist_{INS}(\hat{\delta})  = \min\left\{\dist_{NWP}(\hat{\delta}) ,\dist_{CHAIN}(\hat{\delta}) ,\dist_{FIN}(\hat{\delta})  \right\} \text{.}
	\end{equation}

	\section{Relation between the robust strong $\mathcal{H}_{\infty}$-norm and the robust structured complex distance to instability}
	\label{sec:rob_strong_hinf_dist_instab}
	This section establishes the relation between the robust strong $\mathcal{H}_{\infty}$-norm of system \eqref{eq:system_description} and the robust structured complex distance to instability of \eqref{eq:SDEP}. Section~\ref{subsec:relation_preliminary_results} gives some preliminary results.  In Section~\ref{subsec:distance_nonwellposed} we focus on the relation between the robust strong asymptotic $\mathcal{H}_{\infty}$-norm and the robust structured complex distances to non well-posedness and characteristic root chain crossing. Next, Section~\ref{subsec:distance_finite_root} investigates the link between the worst-case gain function at finite frequencies and the robust structured complex distance to finite root crossing. Finally,  Section~\ref{subsec:central_theorem} combines these results and gives some examples.
	\subsection{Preliminary results}
	\label{subsec:relation_preliminary_results}
	We start with some technical lemmas.
	\begin{lemma} 
		For a matrix $A\in \mathbb{C}^{p\times m}$ it holds that
		\label{lem:distance_instab_sigma1}
		\[
		\sigma_1(A)^{-1}  =\min_{\Delta \in \mathbb{C}^{m \times p}}\left\{ \|\Delta\|_{2} : \det\left(I_m -  \Delta A \right) = 0 \right\} 
		\]
		and
		\[
		\sigma_1(A)^{-1}vu^{H}  \in \argmin_{\Delta \in \mathbb{C}^{m \times p}}\{\|\Delta\|_2 : \det\left(I_m - \Delta A \right) = 0 \} \text{.}
		\]
		where $u$ and $v$ are respectively the left and right singular vectors of $A$ associated with $\sigma_1(A)$, the largest singular value of $A$.
	\end{lemma}
	\begin{proof}
		See for example \cite{Packard1993}.
	\end{proof}
	\begin{lemma}
		\label{lem:SDEP_rob_strong_stable_tf_internal_stable}
		The robust structured complex distance to instability of \eqref{eq:SDEP} is non-zero if and only if the characteristic roots of $I\lambda-\tilde{A}_0(\delta)-\sum_{k=1}^{K}\tilde{A}_k(\delta) e^{-\lambda \tau_k}$ lie in the open left half-plane for all $\delta \in \hat{\delta}$ .
	\end{lemma}
	\begin{proof}
		Because
		\[
		{U_{\mathcal{N}}}^{H}\tilde{P}_0(\delta,\mathbf{0}_{m\times p}) V_{\mathcal{N}}
		= \begin{bmatrix}
		-I & 0 \\
		\tilde{D}_0(\delta) & -I
		\end{bmatrix}
		\]
		is non-singular for all $\delta \in \hat{\delta}$, it follows from Lemma~\ref{lem:condition_wellposed} that the robust structured complex distance to non well-posedness is non-zero. Furthermore, it is easy to verify that
		\[
		\alpha_{D,s}^{\mathrm{ps}}(\hat{\delta},0,\vec{\tau}) = -\infty
		\]
		and thus also the robust structured complex distance to a root chain crossing is non zero.	Finally, one can show that
		\[
		\alpha^{\mathrm{ps}}(\hat{\delta},0,\vec{\tau}) = \max_{\lambda \in \mathbb{C}}\left\{\Re\left(\lambda\right) : \exists \delta \in \hat{\delta} \text{ such that } \det\left(I\lambda-\tilde{A}_0(\delta)-\sum_{k=1}^{K}\tilde{A}_k(\delta) e^{-\lambda \tau_k}\right) = 0 \right\}
		\text{,}\]
		from which the lemma follows.
	\end{proof}
	By Assumption~\ref{asm:robust_internally_stable} this means that all characteristic matrices examined in the remainder of this paper have a positive, non-zero robust structured complex distance to instability.
%
	\begin{lemma}
		\label{lem:relation_tf_ps_SDEP}
		Assume that the complex number $s$ is not a characteristic root of $I\lambda-\tilde{A}_0(\delta^{\star})-\sum_{k=1}^{K}\tilde{A}_k(\delta^{\star}) e^{-\lambda \tau_k}$. There exists a $\Delta \in \mathbb{C}^{m \times p}$ with $\|\Delta\|_2\leq \epsilon$ such that $s$ is a characteristic root of $M(\lambda;\delta^{\star},\Delta,\vec{\tau})$ if and only if $\sigma_1\Big(T(s;\delta^{\star},\vec{\tau})\Big) \geq \epsilon^{-1}$. Furthermore, $s$ is a characteristic root of $M(\lambda;\delta^{\star},\sigma_{1}\big(T(s;\delta^{*},\vec{\tau})\big)^{-1} v u^{H},\vec{\tau})$, where $u$ and $v$ are respectively the left and right singular vectors of $T(s;\delta^{*},\vec{\tau})$ associated with its largest singular value.
	\end{lemma}
	\begin{proof} 
			A complex number $s$ is a characteristic root of $M(\lambda;\delta^{\star},\Delta,\vec{\tau})$ if and only if
			\[
			\begin{aligned}
			\det\big(M(s;\delta^{\star},\Delta,\vec{\tau})\big)& \scalebox{0.75}{=} \det\left(
			\renewcommand{\arraystretch}{1.3}\left[
			\begin{array}{c;{1pt/2pt}cc}
			I s \minus \tilde{A}_0(\delta^{\star})\minus \sum\limits_{k=1}^{K} \tilde{A}_k(\delta^{\star}) e^{\minus s\tau_k}& \minus \tilde{B}_0(\delta^{\star})-\sum\limits_{k=1}^{K} \tilde{B}_k(\delta^{\star}) e^{\minus s\tau_k} & 0 \\ \hdashline[1pt/2pt]
			0 & I & \minus\Delta \\ 
			\minus \tilde{C}_0(\delta^{\star}) \minus \sum\limits_{k=1}^{K} \tilde{C}_k(\delta^{\star})e^{\minus s \tau_k}& \minus \tilde{D}_0(\delta^{\star}) \minus \sum\limits_{k=1}^{K}  \tilde{D}_k(\delta^{\star})e^{\minus s\tau_k} & I
			\end{array}\right] \right) \\ &= 0\text{.}
			\end{aligned}
			\]
			Because $s I-\tilde{A}_0(\delta^{\star}) - \sum_{k=1}^{K}\tilde{A}_{k}(\delta^{\star})e^{-s\tau}$ is invertible, this last expression can be rewritten, using Schur's determinant lemma for block partitioned matrices, in the following form:
			\[
			\det\left(\begin{bmatrix}
			I & -\Delta \\
			-T(s;\delta^{\star},\vec{\tau}) & I
			\end{bmatrix}\right) =\det\left(I-\Delta T(s;\delta^{\star},\vec{\tau})\right)  = 0 \text{.}
			\] 
			From Lemma~\ref{lem:distance_instab_sigma1} it follows that there exists a $\Delta \in \mathbb{C}^{m\times p}$ with $\|\Delta\|_2 \leq \epsilon$ such that this condition is met if and only if
			\[
			\Big(\sigma_{1}\left(T(s;\delta^{\star},\vec{\tau})\right)\Big)^{-1} \leq \epsilon \text{.}
			\]
			 And if this last condition is met, it follows from the second part of Lemma~\ref{lem:distance_instab_sigma1} than one can choose $\Delta = \sigma_{1}\big(T(s;\delta^{*},\vec{\tau})\big)^{-1} v u^{H}$.
	\end{proof}


	\subsection{Link between the asymptotic transfer function and the robust structured complex distances to non well-posedness and characteristic root chain crossing}
	\label{subsec:distance_nonwellposed}
	We start with a characterisation of the robust structured complex distance to non well-posedness in terms of the delay free direct feed-through term of system  \eqref{eq:system_description}.
	\begin{proposition}
		\label{lem:expression_dist_nwp}
		It holds that
		\[
		\dist_{NWP}^{-1}(\hat{\delta}) = \max\limits_{\delta \in \hat{\delta}}\left\{\sigma_{1}\left(\tilde{D}_0(\delta)\right)\right\} \text{.}
		\]
	\end{proposition}
	\begin{proof} 
		It follows from Lemma~\ref{lem:condition_wellposed} that characteristic matrix \eqref{eq:SDEP} is non well-posed if and only if \[{U_{\mathcal{N}}}^{H}\Big(\tilde{P}_0(\delta,\Delta)\Big)V_{\mathcal{N}} =
			\begin{bmatrix}
			-I & \Delta \\
			\tilde{D}_0(\delta) & -I
			\end{bmatrix}
			\] is singular. Using the Schur-Banachiewicz inversion formula for block partitioned matrices, we can rewrite this condition as: Characteristic matrix \eqref{eq:SDEP} is non well-posed if and only if $I-\Delta\tilde{D}_0(\delta)$ is singular. The robust structured complex distance to non well-posedness is thus equal to
		\begin{align*}
			\dist_{NWP}(\delta) &= \begin{cases}
			+\infty\text{, }\det\left(I-\Delta\tilde{D}_0(\delta)\right) \neq 0 \text{ for all } \Delta \in \mathbb{C}^{m \times p}\text{ and } \delta \in \hat{\delta} \\
			\min_{\substack{\delta \in  \hat{\delta} \\ \Delta \in \mathbb{C}^{m\times p}} }\{\|\Delta\|_2: \det\left(I-\Delta\tilde{D}_0(\delta)\right) = 0  \} \text{, otherwise}
			\end{cases}
			\intertext{Using Lemma~\ref{lem:distance_instab_sigma1}, one finds}
			\dist_{NWP}(\delta) &= \begin{cases}
			+ \infty & \sigma_{1}\Big(\tilde{D}_0(\delta)\Big) = 0 \text{ for all }\delta\in\hat{\delta}\\
			\min\limits_{\delta \in  \hat{\delta}}\left\{ \sigma_{1}\Big(\tilde{D}_0(\delta)\Big)^{-1}\right\}& \text{otherwise}
			\end{cases} \\
			&= \begin{cases}
			+ \infty & \sigma_{1}\Big(\tilde{D}_0(\delta)\Big) = 0 \text{ for all }\delta\in\hat{\delta}\\
			\left(\max\limits_{\delta \in  \hat{\delta}} \left\{\sigma_{1}\Big(\tilde{D}_0(\delta)\Big)\right\}\right)^{-1}& \text{otherwise}
			\end{cases}
		\end{align*}
	\end{proof}
	Next we derive a condition for a finite robust structured complex distance to a characteristic root chain crossing in terms of the robust strong asymptotic $\mathcal{H}_{\infty}$-norm of system \eqref{eq:system_description}.
	\begin{lemma}
		\label{lem:link_strong_chain_tranferf}
		The robust structured complex distance to a characteristic root chain crossing of \eqref{eq:SDEP} is finite if and only if
		\[
		\max_{\delta \in \hat{\delta}}\left\{\sigma_1\left(\tilde{D}_0(\delta)\right)\right\} < |||T_a(\cdot;\cdot,\vec{\tau})|||_{\mathcal{H}_{\infty}}^{\hat{\delta}}
		\]
	\end{lemma}
	\begin{proof}
			By Property~\ref{prop:definition_strong_spectral_abscissa} and under Assumption~\ref{asm:robust_internally_stable}, the robust structured complex distance to a characteristic root chain crossing of \eqref{eq:SDEP} is finite, if and only if there exist $\delta^{\star} \in \hat{\delta}$, $\vec{\theta}^{\star}\in [0,2\pi)^{K}$ and $\Delta^{\star}\in\mathbb{C}^{m\times p}$ with $\|\Delta^{\star}\|_2  < \dist_{NWP}$  such that  
			\[
			 \det\left(I+\sum_{k=1}^{K}\left({U_{\mathcal{N}}}^{H} \tilde{P}_0(\delta^{\star},\Delta^{\star})V_{\mathcal{N}}\right)^{-1}{U_{\mathcal{N}}}^{H}\tilde{P}_k(\delta^{\star}) V_{\mathcal{N}}e^{\jmath\theta_k^{\star}}\right) = 0 \text{.}
			\]
			Because ${U_{\mathcal{N}}}^{H}\tilde{P}_0(\delta^{\star},\Delta^{\star})V_{\mathcal{N}}$ is non-singular ($\|\Delta^{\star}\|_2< \dist_{NWP}$), this last condition is equivalent with
			\[
			\det\left({U_{\mathcal{N}}}^{H}\tilde{P}_0(\delta^{\star},\Delta^{\star})V_{\mathcal{N}}+\sum_{k=1}^{K}{U_{\mathcal{N}}}^{H}\tilde{P}_{k}(\delta^{\star}) V_{\mathcal{N}}e^{\jmath\theta_k^{\star}}\right) = 0 \text{.}
			\]
			Plugging in the definitions of $\tilde{P}_0,\dots,\tilde{P}_k$ and using Schur's determinant lemma for block partitioned matrices the right hand side of the condition reduces to: there exist $\delta^{\star} \in \hat{\delta}$, $\vec{\theta}^{\star}\in [0,2\pi)^{K}$ and $\Delta^{\star}$ with $\|\Delta^{\star}\|_2  < \dist_{NWP}$ such that
			\[
			\det\left(I-\Delta^{\star}\left(\tilde{D}_0(\delta^{\star})+\sum_{k=1}^{K}\tilde{D}_k(\delta^{\star}) e^{\jmath\theta_k^{\star}}\right)\right) = 0\text{.}
			\]
			The lemma follows from \eqref{eq:strong_H_infnorm_asymptotic}, Lemma~\ref{lem:distance_instab_sigma1} and Proposition~\ref{lem:expression_dist_nwp}.
	\end{proof}
	The following lemma gives a lower bound for the robust strong asymptotic $\mathcal{H}_{\infty}$-norm.
	\begin{lemma}
	\label{lem:dist_NWP_dist_CHAIN}
	It holds that
	\[
	\max_{\delta \in \hat{\delta}}\left\{\sigma_{1}\Big(\tilde{D}_0(\delta)\Big) \right\}\leq  \max_{\delta \in  \hat{\delta}}\max_{\theta \in [0,2\pi)^{K}} \left\{ \sigma_{1}\left(\tilde{D}_0(\delta)+\sum_{k=1}^{K}
	\tilde{D}_k(\delta) e^{\jmath\theta_k}\right)\right\} = |||T_a(\cdot;\cdot,\vec{\tau})|||_{\mathcal{H}_{\infty}}^{\hat{\delta}} \text{.}\]
\end{lemma}
\begin{proof}
	Consider the matrix-valued function $\mathbb{C} \ni s  \mapsto  D_0 + \sum_{k=1}^{K} D_k e^{-s} \in \mathbb{C}^{p\times m}$, of which every entry is analytic and bounded for $\Re(s)\geq 0$. From \cite{Boyd1984} it follows that $s\mapsto \sigma_{1}(D_0 + \sum_{k=1}^{K} D_k e^{-s})$ attains it maximum over the closed right-halfplane on $\Re(s) = 0$. And thus
	\[
	\begin{aligned}
	\sigma_{1}(D_0) =  \lim\limits_{s\to \infty} \sigma_{1}\left(D_0 + \sum_{k=1}^{K} D_k e^{\scalebox{0.7}[1.0]{\( - \)}s}\right)&\leq  \max_{\omega\in\mathbb{R}} \sigma_{1}\left(D_0 + \sum_{k=1}^{K} D_k e^{\scalebox{0.7}[1.0]{\( - \)}\jmath\omega}\right) \\
	& \leq \max_{\theta \in [0,2\pi)^{K}}\sigma_{1}\left(D_0 + \sum_{k=1}^{K} D_k e^{\jmath\theta_k}\right) \text{.}
	\end{aligned}	
	\]
\end{proof}
	By combining these results we get an expression for the robust strong asymptotic $\mathcal{H}_{\infty}$-norm in terms of the robust structured distances to non well-posedness and a characteristic root chain crossing.
	\begin{proposition}
		\label{prop:dist_NWP_CHAIN_asympt}
		The robust strong asymptotic $\mathcal{H}_{\infty}$-norm of \eqref{eq:system_description}  is equal to the reciprocal of the minimum of the robust structured complex distance to non well-posedness and the robust structured complex distance to a characteristic root chain crossing, ie.
		\[
		|||T_a(\cdot;\cdot,\vec{\tau})|||_{\mathcal{H}_\infty}^{\hat{\delta}} =  \min\{\dist_{NWP}(\hat{\delta}),\dist_{CHAIN}(\hat{\delta})\}^{-1} \text{.}
		\]
	\end{proposition}
	\begin{proof}
		
		First we consider the case where $\dist_{CHAIN}(\hat{\delta})<\dist_{NWP}(\hat{\delta})$. Using a similar idea as in the proof of Lemma~\ref{lem:link_strong_chain_tranferf} it can be shown that for $\epsilon \in [0,\dist_{NWP}(\hat{\delta}))$ $\alpha_{D,s}^{\mathrm{ps}}(\hat{\delta},\epsilon,\vec{\tau}) \geq 0$ if and only if
		\[
		|||T_a(\cdot;\cdot,\vec{\tau})|||_{\mathcal{H}_{\infty}}^{\hat{\delta}} \geq \epsilon^{-1} \text{.}
		\]
		The smallest $\epsilon$ such this last condition is fulfilled is equal to ${|||\tilde{T}_a(\cdot;\cdot,\vec{\tau})|||_{\mathcal{H}}^{\hat{\delta}}}^{-1}$ and thus \[|||\tilde{T}_a(\cdot;\cdot,\vec{\tau})|||_{\mathcal{H}}^{\hat{\delta}} = \dist_{CHAIN}(\hat{\delta})^{-1}\].
		Next we consider the case where the robust structured complex distance to a characteristic root chain crossing is not finite. It follows from Lemma~\ref{lem:link_strong_chain_tranferf} that $\max_{\delta\in\hat{\delta}}\left\{\sigma_1\left(\tilde{D}_0(\delta)\right)\right\} \geq |||T_a(\cdot;\cdot,\vec{\tau})|||_{\mathcal{H}_{\infty}}^{\hat{\delta}}$. But by Lemma~\ref{lem:dist_NWP_dist_CHAIN} we have $\allowbreak \max_{\delta\in\hat{\delta}} \left\{\sigma_1\left(\tilde{D}_0(\delta)\right)\right\} \leq |||T_a(\cdot;\cdot,\vec{\tau})|||_{\mathcal{H}_{\infty}}^{\hat{\delta}}$. 
		Thus in this case
		\[
		|||T_a(\cdot;\cdot,\vec{\tau})|||_{\mathcal{H}_{\infty}}^{\hat{\delta}} =
		\max_{\delta\in\hat{\delta}}\left\{\sigma_1\left(\tilde{D}_0(\delta)\right)\right\} = \dist_{NWP}(\hat{\delta})^{-1} \text{.}
		\]
	\end{proof}

	\subsection{Link between the worst-case gain function at finite frequencies and the robust structured complex distance to finite root crossing}
	\label{subsec:distance_finite_root}
	The previous subsection established a relation between the robust strong asymptotic $\mathcal{H}_{\infty}$-norm and the robust structured complex distances to non well-posedness and a characteristic root chain crossing. In this subsection the link between the worst-case gain function (as defined in \eqref{eq:worst_case_transfer_function}) and the robust structured complex distance to finite root crossing is examined.
	\begin{lemma}
		\label{lem:link_finite_root_crossing_transferf}
		The robust structured complex distance to finite root crossing is finite if and only if
		system \eqref{eq:system_description} attains its robust strong $\mathcal{H}_{\infty}$-norm at a finite frequency (ie. $||T(\cdot;\cdot,\vec{\tau})||_{\mathcal{H}_{\infty}}^{\hat{\delta}}>|||T_a(\cdot;\cdot,\vec{\tau})|||_{\mathcal{H}_{\infty}}^{\hat{\delta}}$). In such a case it holds 
		\[ |||T(\cdot;\cdot,\vec{\tau})|||_ {\mathcal{H}_{\infty}}^{\hat{\delta}} =  \|T(\cdot;\cdot,\vec{\tau})\|_ {\mathcal{H}_{\infty}}^{\hat{\delta}} = \dist_{FIN}^{-1}(\hat{\delta}) \text{.}
		\]
		
	\end{lemma}
	\begin{proof} {\ \\ }
		$\Rightarrow$ It follows from Remark~\ref{remark:critical_epsilon} and Assumption~\ref{asm:robust_internally_stable} that if the robust structured complex distance to finite root crossing is finite, there exist $\delta^{\star} \in \hat{\delta}$, $\Delta^{\star} \in \mathbb{C}^{m \times p}$ with $\|\Delta^{\star}\|_2\leq \dist_{FIN}(\hat{\delta})$ and a finite $\omega\in\mathbb{R}$ such that $\jmath\omega$ is a characteristic root of $M(\lambda;\delta^{\star},\Delta^{\star},\vec{\tau})$. By Lemma~\ref{lem:relation_tf_ps_SDEP}, this means that $$\sigma_1(T(\jmath\omega;\delta^{\star},\vec{\tau}))\geq (\dist_{FIN}(\hat{\delta}))^{-1} > \left(\min\{\dist_{NWP}(\hat{\delta}),\dist_{CHAIN}(\hat{\delta}) \}\right)^{-1}\text{.}$$ Using Proposition~\ref{prop:dist_NWP_CHAIN_asympt} one finds that $$\|T(\cdot;\cdot,\vec{\tau})\|_{\mathcal{H}_{\infty}}^{\hat{\delta}} \geq\sigma_1(T(\jmath\omega;\delta^{\star},\vec{\tau})) > |||T_a(\cdot;\cdot,\vec{\tau})|||_{\mathcal{H}_{\infty}}^{\hat{\delta}}\text{.}$$ \\
		$\Leftarrow$ If the robust strong $\mathcal{H}_{\infty}$-norm is attained at a finite frequency then there exists $\delta \in \hat{\delta}$ and
			$\omega \in \mathbb{R}$ such that $\sigma_1(T(\jmath\omega;\delta,\vec{\tau})) >  |||T_a(\cdot;\cdot,\vec{\tau})|||_{\mathcal{H}_{\infty}}^{\hat{\delta}}$. From Lemma~\ref{lem:relation_tf_ps_SDEP} and Proposition~\ref{prop:dist_NWP_CHAIN_asympt} it follows that there exists a $\Delta$ with $\|\Delta\|_2 =\sigma_1(T(\jmath\omega;\delta,\vec{\tau}))^{-1}< \min\left\{\dist_{NWP}(\hat{\delta}),\dist_{CHAIN}(\hat{\delta})\right\}$ such that characteristic matrix $M(\lambda;\delta,\Delta,\vec{\tau})$ has a characteristic root at $\jmath \omega$. \\ 
		$\|T(\cdot;\cdot,\vec{\tau})\|_{\mathcal{H}_{\infty}}^{\hat{\delta}} = \dist_{FIN}(\hat{\delta})^{-1}$ is found using Lemma~\ref{lem:relation_tf_ps_SDEP} and by maximising $\sigma_1(T(\jmath\omega;\delta,\vec{\tau}))$ over all $\delta \in \hat{\delta}$ and $\omega\in \mathbb{R}$.
	\end{proof}

	\subsection{Main theoretical result}
	\label{subsec:central_theorem}
	In this subsection the results of the two previous subsections are combined to characterise the robust strong $\mathcal{H}_{\infty}$-norm in terms of the robust structured complex distance to instability. Subsequently some examples are given. 
	\begin{theorem}
		\label{the:main_theoretical_result}
		The robust strong $\mathcal{H}_{\infty}$-norm of an internally exponentially stable system of form \eqref{eq:system_description} is equal to the reciprocal of the robust structured complex distance to instability of characteristic matrix \eqref{eq:SDEP}, ie.
		\[
		|||T(\cdot;\cdot,\vec{\tau})|||_{\mathcal{H}_{\infty}}^{\hat{\delta}} = \frac{1}{\dist_{INS}(\hat{\delta})}
		\]
	\end{theorem}
	\begin{proof}
		If the robust strong $\mathcal{H}_{\infty}$-norm is attained at finite frequencies then the result follows from Lemma~\ref{lem:link_finite_root_crossing_transferf}. Otherwise the result follows from Proposition~\ref{prop:dist_NWP_CHAIN_asympt}.
	\end{proof}	
	As mentioned in Section~\ref{sec:introduction}, the robust strong $\mathcal{H}_{\infty}$-norm of system \eqref{eq:system_description} is either equal to the robust strong asymptotic $\mathcal{H}_{\infty}$-norm or to the maximum of the worst-case gain function. In Section~\ref{subsec:distance_nonwellposed} it was shown that the former is related to the robust structured complex distances to non well-posedness and a characteristic root chain crossing. Section~\ref{subsec:distance_finite_root} proved that the latter relates with the robust structured complex distance to finite root crossing. The following examples illustrate this duality in more detail.

		\begin{example}
		\label{example:finite_crossing}
		In this first example we consider the following uncertain system
		\begin{equation}
		\label{eq:illustration_link_2_system}
		\left\{ 
		\begin{array}{rcl}
		\dot{x}(t) & = & (\minus 7+3\delta_1) x(t)  +(-5+2\delta_2)x(t-1) + 4 w(t) \\ [2ex]
		z(t) & = & (2\minus 2\delta_2) x(t) +  (1+\delta_1)w(t) +  w(t-1)
		\end{array}
		\right.
		\end{equation}
		where $\delta_1$ and $\delta_2$ are respectively confined to $|\delta_1| \leq 0.15$ and $|\delta_2| \leq 0.2$, and its corresponding characteristic matrix:
		\begin{equation}
		\label{eq:illustration_link_2_SDEP}
		\begin{bmatrix}
		\lambda & 0 & 0 \\
		0 & 0 & 0 \\
		0 & 0 & 0
		\end{bmatrix}  - \begin{bmatrix}
		(-7+3\delta_1) & 4  & 0 \\
		0  & -1 & \Delta \\
		(2-2\delta_2) & (1+\delta_1)  & -1
		\end{bmatrix}
		- \begin{bmatrix}
		(-5+2\delta_2)& 0 & 0 \\
		0 & 0 & 0 \\
		0 & 1 & 0 
		\end{bmatrix}e^{-\lambda} \text{.}
		\end{equation}		
	 First we examine the robust structured complex distance to instability of \eqref{eq:illustration_link_2_SDEP}. Subsequently, we illustrate the relation of this distance measure with the behaviour of system  \eqref{eq:illustration_link_2_system}.

	Proposition~\ref{lem:expression_dist_nwp} gives us an expression for the robust structured complex distance to non well-posedness: 
		\[\dist_{NWP}(\hat{\delta}) = \Big(\max_{|\delta_1|\leq 0.15} |1+\delta_1|\Big)^{-1} = 1/1.15 = 0.8696 \text{.}\] 
	To find the robust structured complex distances to a characteristic root chain crossing and finite root crossing we plot $\alpha^{\mathrm{ps}}(\hat{\delta},\epsilon,\vec{\tau})$ and $\alpha_{D,s}^{\mathrm{ps}}(\hat{\delta},\epsilon,\vec{\tau})$ as a function of $\epsilon$ \Big(for $\epsilon<\dist_{NWP}(\hat{\delta})$\Big) in Figure~\ref{fig:illustration_link2_psa}:  
\[
\begin{aligned}
\dist_{CHAIN}(\hat{\delta}) &= 0.4651 \\
\dist_{FIN}(\hat{\delta}) &= 0.1663 \text{.}
\end{aligned}
\]	
The robust structured complex distance to instability is thus equal to $0.1663$ and the loss of strong stability is caused by a finite number of characteristic roots moving into the closed right-half plane. This is illustrated in Figures~\ref{fig:illustration_link2_spectrum_fin} and \ref{fig:illustration_link2_pseudo_spectrum}. Figure~\ref{fig:illustration_link2_spectrum_fin} shows the spectrum of  \eqref{eq:illustration_link_2_SDEP} for the perturbations associated with the loss of strong stability and Figure~\ref{fig:illustration_link2_pseudo_spectrum} shows its $(\hat{\delta},\epsilon)$-pseudo-spectrum for $\epsilon$ smaller than, equal to, and larger than the distance to instability. We observe that the loss of strong stability is caused by characteristic roots moving into the closed right-half plane at $s=\pm\jmath 2.734$.

	Next, we examine how these distance measures relate to the robust strong $\mathcal{H}_{\infty}$-norm of system \eqref{eq:illustration_link_2_system}. The robust strong asymptotic $\mathcal{H}_{\infty}$-norm follows from \eqref{eq:strong_H_infnorm_asymptotic}: 
	\[
	\begin{aligned}
	|||T_a(\cdot;\cdot,\vec{\tau})|||_{\mathcal{H}_{\infty}}^{\hat{\delta}} &= 
	\max_{|\delta_1|\leq 0.15} \max_{\theta\in [0,2\pi)} \Big|(1+\delta_1)+1e^{\jmath\theta}\Big| \\ &= 2.15 \\ \Big(&= \min\{\dist_{NWP}(\hat{\delta}),\dist_{CHAIN}(\hat{\delta})\}^{-1}\Big)\text{.}
	\end{aligned}	
	\]
	 Figure~\ref{fig:illustration_link2_supT} plots the worst-case gain function. This function attains a maximum value of $6.012$ ($=\dist_{FIN}(\hat{\delta})^{-1}$, indicated in magenta) at a finite frequency ($\omega = 2.734$). Furthermore, for this example the robust strong asymptotic $\mathcal{H}_{\infty}$-norm can also be deduced from this worst-case gain function (indicated in red) as the system has only one delay which means that: 
	\[|||T_a(\cdot;\cdot,\vec{\tau})|||_{\mathcal{H}_{\infty}}^{\hat{\delta}}= \|T_a(\cdot;\cdot,\vec{\tau})\|_{\mathcal{H}_{\infty}}^{\hat{\delta}} = \limsup_{\omega \to \infty} \max_{\delta\in \hat{\delta}} \sigma_1\Big(T(\jmath\omega;\delta,\vec{\tau})\Big) = 2.15\text{.}\]
	The robust strong $\mathcal{H}_{\infty}$-norm is thus equal to $6.012$ and is equal to the maximum of the worst-case gain functions. On figure \ref{fig:illustration_link2_supT} we have also indicated $\dist_{NWP}(\hat{\delta})^{-1} = \max_{\delta \in \hat{\delta}} \left\{\sigma_1\left(\tilde{D}_0(\delta)\right)\right\}$. One observes that as $\omega$ goes to infinity, the worst-case gain function oscillates around this value.	
		\begin{figure}[!hbtp]
			\centering
			\begin{minipage}[t]{0.48\textwidth}
				\centering
				\includegraphics[width=\linewidth]{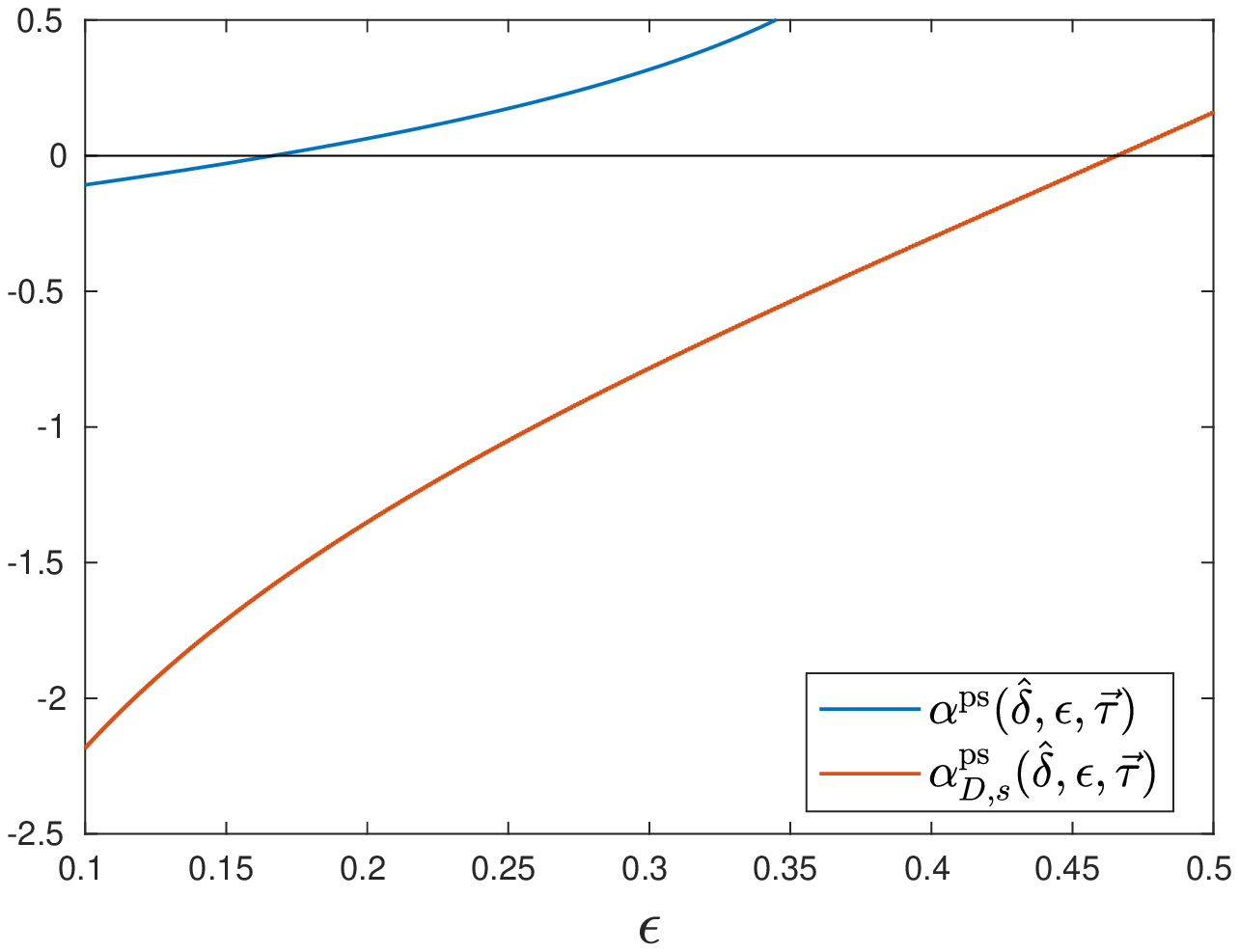}
				\captionsetup{width=.8\linewidth}
				\caption{$\alpha^{\mathrm{ps}}(\hat{\delta},\epsilon,\vec{\tau})$ (blue) and $\alpha_{D,s}^{\mathrm{ps}}(\hat{\delta},\epsilon,\vec{\tau})$ (red) of characteristic matrix \eqref{eq:illustration_link_2_SDEP} in function of $\epsilon$.}
				\label{fig:illustration_link2_psa}
			\end{minipage}
			\begin{minipage}[t]{0.48\linewidth}
				\centering
				\includegraphics[width=\linewidth]{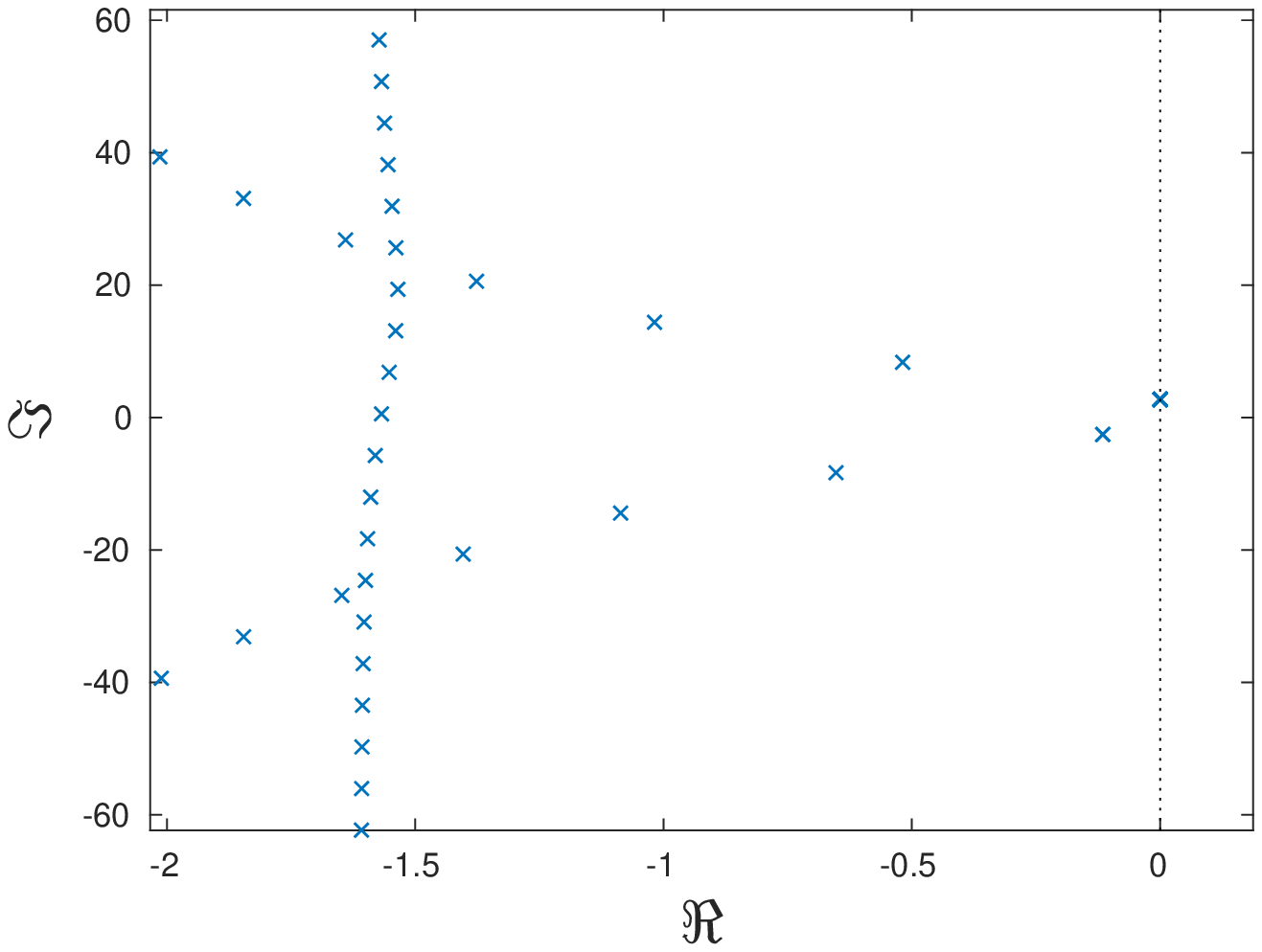}
				\caption{The spectrum of characteristic matrix \eqref{eq:illustration_link_2_SDEP} for the perturbations associated with the loss of strong stability: $\Delta=0.1663e^{\jmath0.4065}$, $\delta_1 = 0.15$ and $\delta_2 = -0.2$.}
				\label{fig:illustration_link2_spectrum_fin}
			\end{minipage}
		\end{figure}
		\begin{figure}[!hbtp]
			\centering
			\begin{subfigure}{0.48\linewidth}
				\centering
				\includegraphics[width=\linewidth]{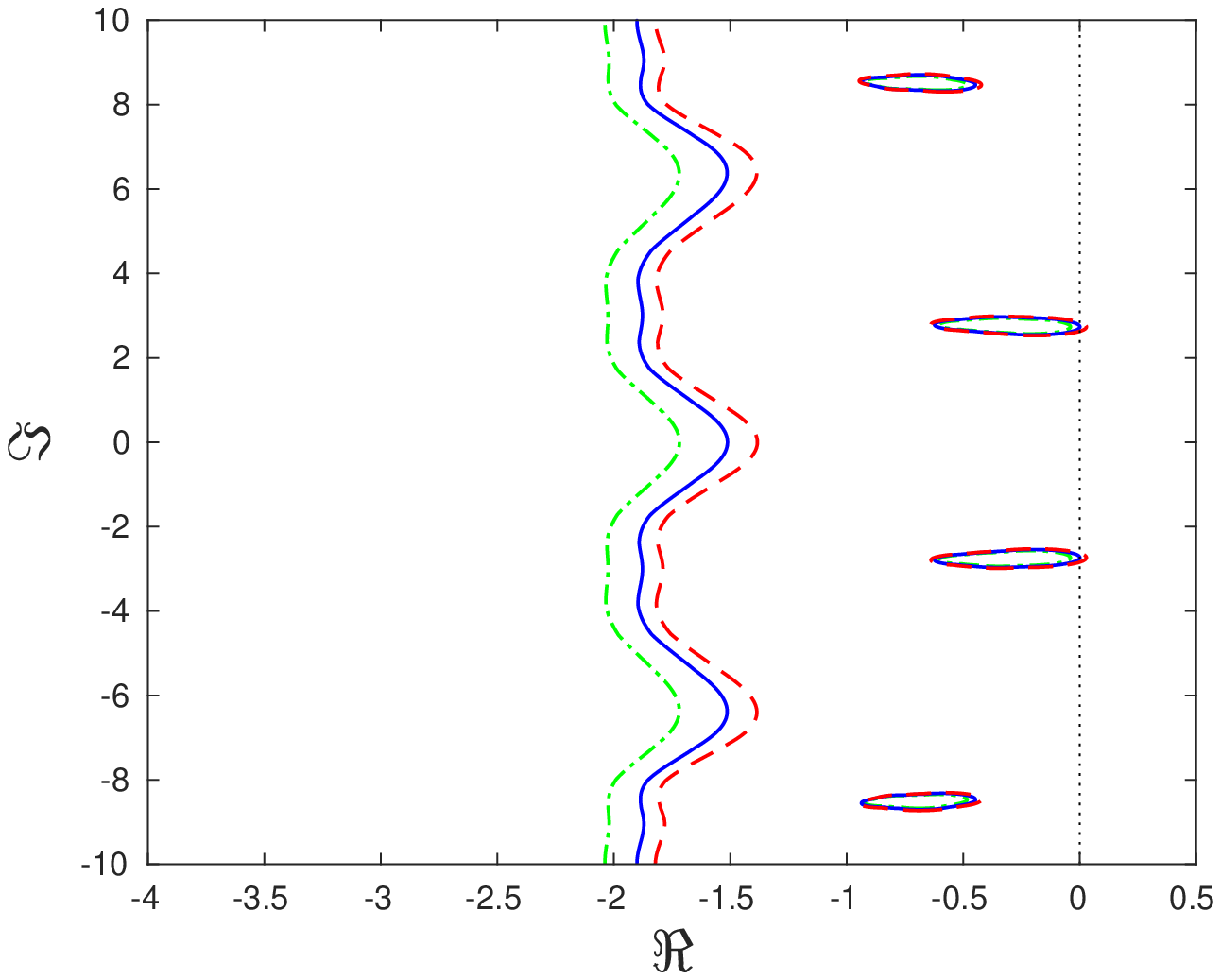}
			\end{subfigure}
			\begin{subfigure}{0.48\linewidth}
				\centering
				\includegraphics[width=\linewidth]{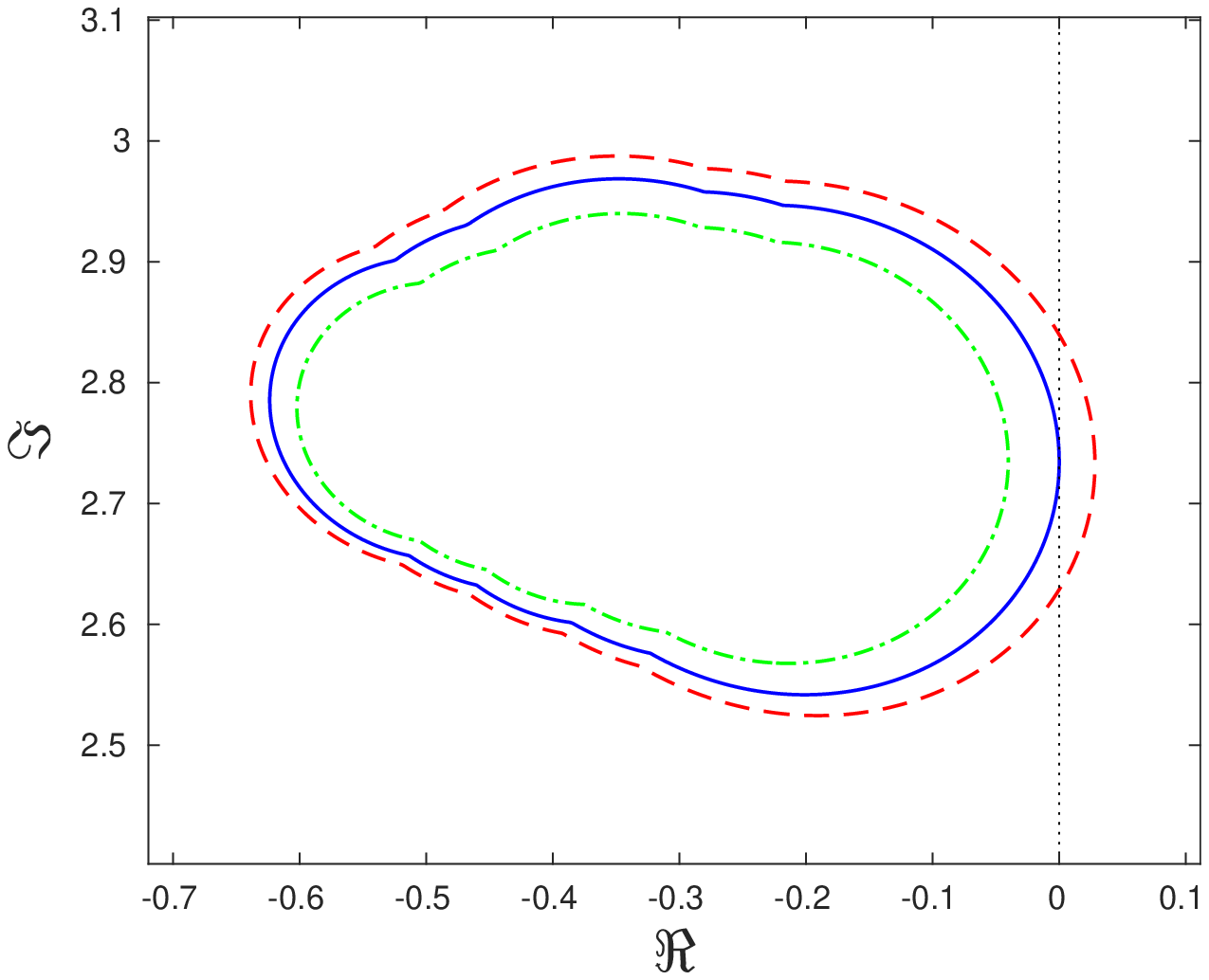}
			\end{subfigure}
			\caption{The $(\hat{\delta},\epsilon)$-pseudo-spectrum of characteristic matrix \eqref{eq:illustration_link_2_SDEP} for $\epsilon$ equal to $0.1429$ (green dot dash line), $0.1663$ (blue full line) and $0.2$ (red dashed line).}
			\label{fig:illustration_link2_pseudo_spectrum}
		\end{figure}
		\begin{figure}
				\centering
				\includegraphics[width=0.6\linewidth]{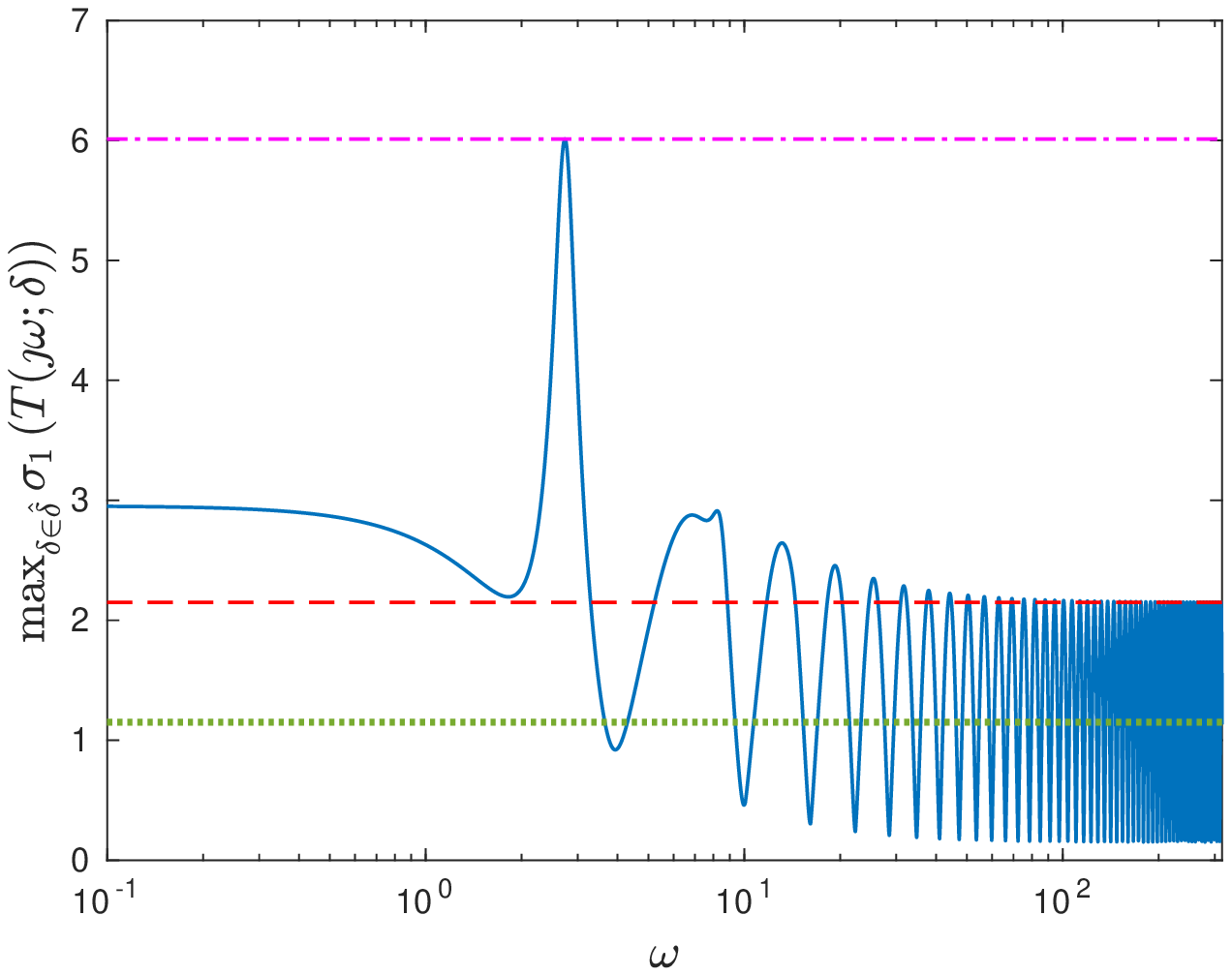}
				\captionsetup{width=\linewidth}
				\caption{The worst-case gain function of system~\eqref{eq:illustration_link_2_system} (blue), $||T(\cdot;\cdot,\vec{\tau})||_{\mathcal{H}_{\infty}}^{\hat{\delta}}$ (magenta dot-dashed), $|||T_a(\cdot;\cdot,\vec{\tau})|||_{\mathcal{H}_{\infty}}^{\hat{\delta}}$ (=$\|T_a(\cdot;\cdot,\vec{\tau})\|_{\mathcal{H}_{\infty}}^{\hat{\delta}}$) (red dashed) and $\dist_{NWP}(\hat{\delta})^{-1}$ (green dotted).}
				\label{fig:illustration_link2_supT}
			\end{figure}
	\end{example}
	\begin{example}
	\label{example:chain_crossing}
	Next we consider the following uncertain system:
	\begin{equation}
		\label{eq:illustration_link3_system}
		\left\{
		\begin{array}{rcl}
		\dot{x}(t) & = & (-3+\delta_2)x(t) + (-1+3\delta_1)x(t-1)+ 4w(t) \\
		z(t) & =& (-2-3\delta_1+2\delta_2) x(t) + (3+\delta_1) w(t)+ (1+\delta_1+\delta_2) w(t-1)
		\end{array}\right.
	\end{equation}
	where $\delta_1$ and $\delta_2$ are respectively confined to $|\delta_1| \leq 0.1$ and $|\delta_2| \leq 0.25$, and its associated characteristic matrix:
	\begin{equation}
	\label{eq:illustration_link3_SDEP}
	\begin{bmatrix}
	\lambda & 0 & 0 \\
	0 & 0 & 0 \\
	0 & 0 & 0
	\end{bmatrix}-
	\begin{bmatrix}
	-3+\delta_2 & 4 & 0 \\
	0 & -1 & \Delta \\
	-2-3\delta_1+2\delta_2 & 3+\delta_1 & -1
	\end{bmatrix}-
	\begin{bmatrix}
	-1+3\delta_1 & 0 & 0 \\
	0 & 0 & 0 \\
	0 & 1+\delta_1+\delta_2 & 0
	\end{bmatrix}e^{-\lambda}\text{.}
	\end{equation}
	As before, we first examine the robust structured complex distance to instability of \eqref{eq:illustration_link3_SDEP}. The robust structured complex distance to non well-posedness follows from Proposition~\ref{lem:expression_dist_nwp}:
	\[\dist_{NWP} = \left(\max_{|\delta_1| \leq 0.1} |3+\delta_1|\right)^{-1} = 1/3.1 = 0.3226\text{.}\]
	The robust structured complex distances to a characteristic root chain crossing and finite root crossing follow from Figure~\ref{fig:illustration_link3_psa}:
	\[
	\begin{aligned}
	\dist_{CHAIN}(\hat{\delta}) &= 0.2247 \\
	\dist_{FIN}(\hat{\delta}) &= +\infty \text{.}
	\end{aligned}
	\]
	Hence, the robust structured complex distance to instability is equal to $0.2247$ and the loss of strong stability is caused by a chain of characteristic roots whose vertical asymptote moves into the closed right-half plane. This is illustrated in Figure~\ref{fig:illustration_link3_spectrum} which shows the spectrum of \eqref{eq:illustration_link3_SDEP} for the perturbations associated with $\alpha_{D,s}^{\mathrm{ps}}(\hat{\delta},\epsilon,\vec{\tau})$ and $\epsilon$ smaller than, equal to and larger than the robust structured complex distance to instability.
	
	Figure~\ref{fig:illustration_link3_supT} shows the worst-case gain function of system \eqref{eq:illustration_link3_system}. In contrast to the previous example, it attains its maximal value (of $4.45$) only at infinity, ie. $\|T(\cdot;\cdot,\vec{\tau})\|_{\mathcal{H}_{\infty}}^{\hat{\delta}} = \|T_a(\cdot;\cdot,\vec{\tau})\|_{\mathcal{H}_{\infty}}^{\hat{\delta}}$. Furthermore, as in the previous example $\|T_a(\cdot;\cdot,\vec{\tau})\|_{\mathcal{H}_{\infty}}^{\hat{\delta}}$ and \[
	\begin{aligned}
	|||T_a(\cdot;\cdot,\vec{\tau})|||_{\mathcal{H}_{\infty}}^{\hat{\delta}} &= \max_{\substack{|\delta_1|\leq 0.1 \\ |\delta_2|\leq 0.25}} \max_{\theta \in [0,2\pi)} \left|3+\delta_1+(1+\delta_1+\delta_2) e^{j\theta}\right| \\
	&= 4.45 \\
	\Big(&= \min\left\{\dist_{NWP}(\hat{\delta}),\dist_{CHAIN}(\hat{\delta})\right\}^{-1}\Big)
	\end{aligned}
	\] coincide. Hence, the robust strong $\mathcal{H}_{\infty}$-norm equals 4.45 and corresponds to the robust (strong) asymptotic $\mathcal{H}_{\infty}$-norm.
			\begin{figure}[!hbtp]
			\centering
			\begin{minipage}[t]{0.48\textwidth}
				\centering
				\includegraphics[width=\linewidth]{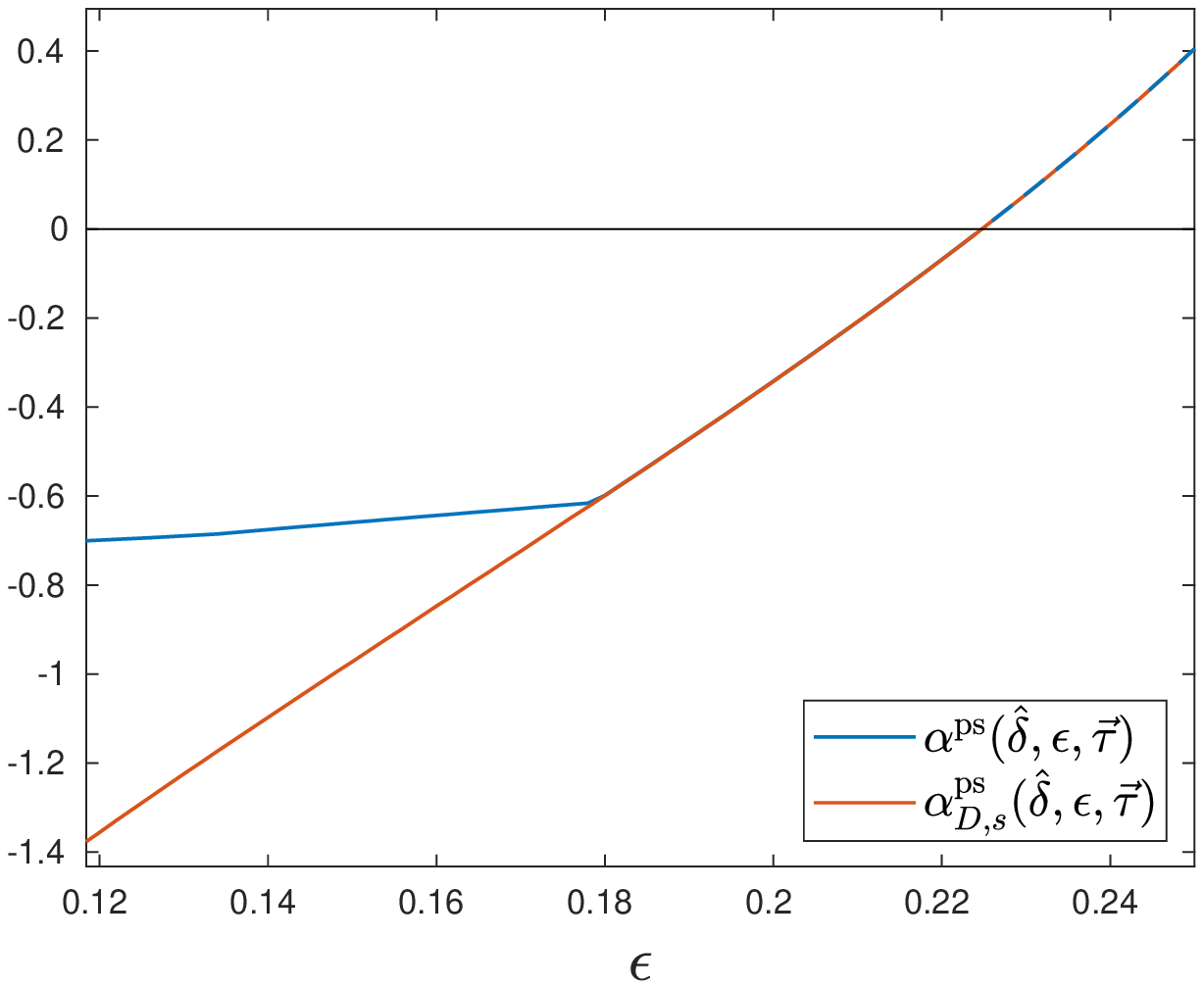}
				\captionsetup{width=.9\linewidth}
				\caption{$\alpha^{\mathrm{ps}}(\hat{\delta},\epsilon,\vec{\tau})$ (blue) and $\alpha_{D,s}^{\mathrm{ps}}(\hat{\delta},\epsilon,\vec{\tau})$ (red) of characteristic matirx \eqref{eq:illustration_link3_SDEP} in function of $\epsilon$.}
				\label{fig:illustration_link3_psa}
			\end{minipage}
			\begin{minipage}[t]{0.48\linewidth}
				\centering
				\includegraphics[width=\linewidth]{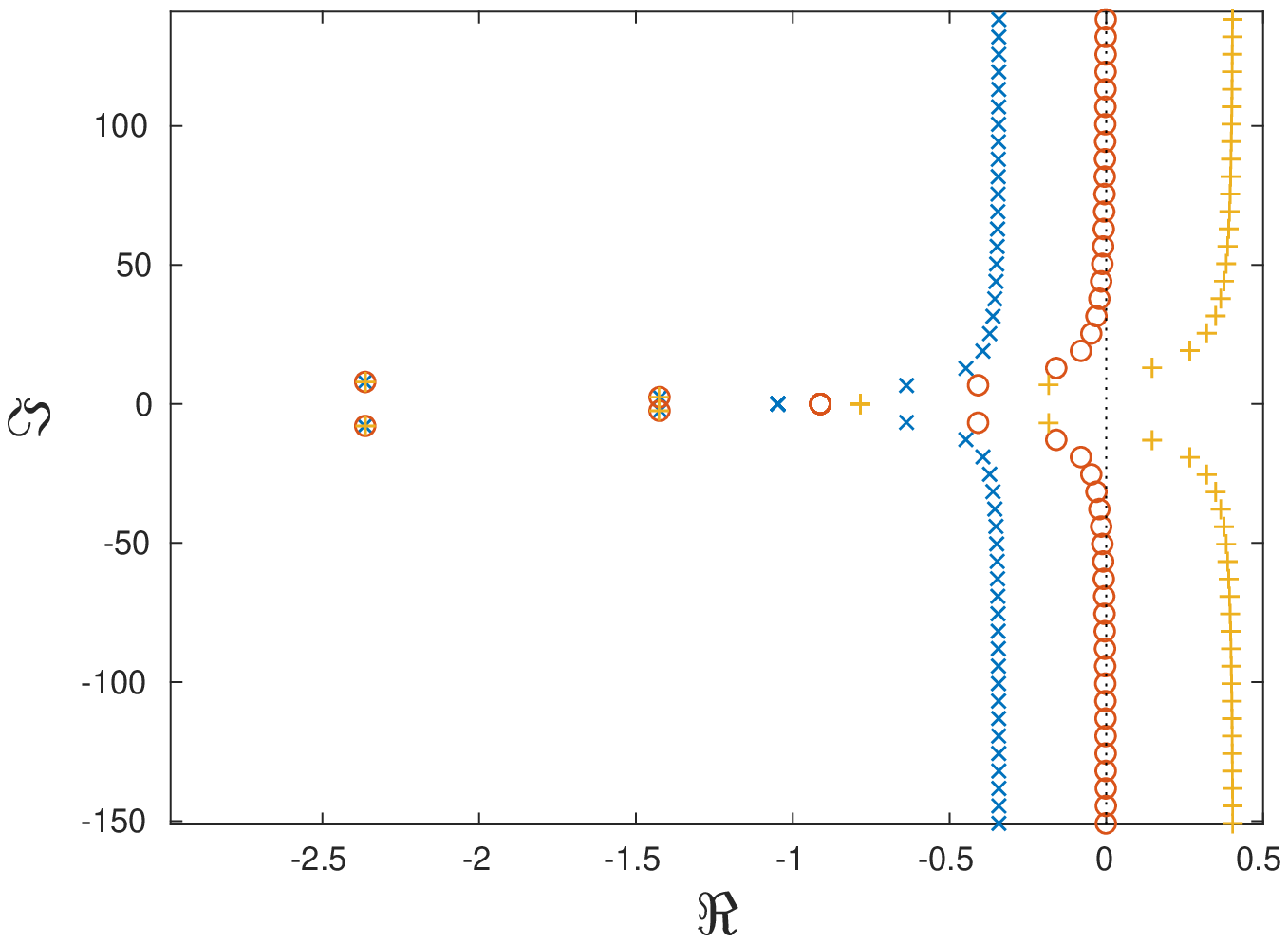}
				\captionsetup{width=.9\linewidth}
				\caption{The spectrum of \eqref{eq:illustration_link3_SDEP} for the perturbations associated with $\alpha_{D,s}^{\mathrm{ps}}(\hat{\delta},\epsilon,\vec{\tau})$ and $\epsilon$ equal to 0.2 (blue x), $\dist_{INS}(\hat{\delta})$ (red o) and 0.25 (yellow +).}
				\label{fig:illustration_link3_spectrum}
			\end{minipage}\hfill
		\end{figure}
	\begin{figure}
				\centering
				\includegraphics[width=0.6\linewidth]{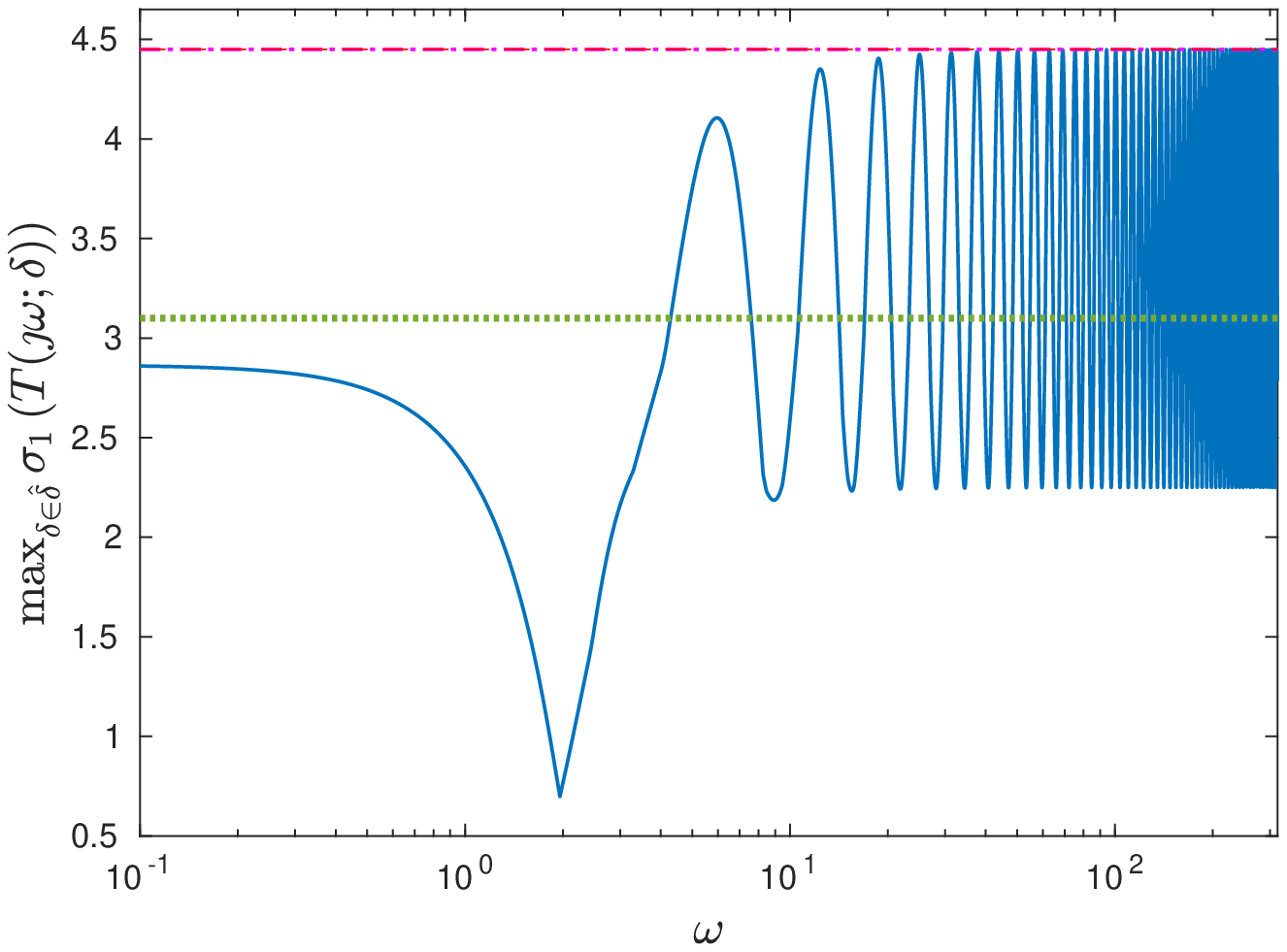}
				\captionsetup{width=\linewidth}
				\caption{The worst-case gain function of system~\eqref{eq:illustration_link3_system} (blue), $||T(\cdot;\cdot,\vec{\tau})||_{\mathcal{H}_{\infty}}^{\hat{\delta}}$ (magenta dot-dashed), $|||T_a(\cdot;\cdot,\vec{\tau})|||_{\mathcal{H}_{\infty}}^{\hat{\delta}}$ (=$\|T_a(\cdot;\cdot,\vec{\tau})\|_{\mathcal{H}_{\infty}}^{\hat{\delta}}$) (red dashed) and $\dist_{NWP}(\hat{\delta})^{-1}$ (green dotted).}
				\label{fig:illustration_link3_supT}
			\end{figure}
	\end{example}
	
	\begin{example}
		\label{example:strong_chain_crossing}
		As third and last example we consider the following uncertain system \big(whose nominal model corresponds to \eqref{eq:example0_system_description}\big):
		\begin{equation}
		\label{eq:illustration_link_1_system}
		\left\{ 
		\begin{array}{rcl}
		\dot{x}(t) & = & (\minus 2+\delta_1) x(t)  +(1+\delta_2)x(t-1)  \minus w(t)  + (\minus 0.5+\delta_1) w(t-2)\\ [2ex]
		z(t) & = & (\minus 2+2\delta_2) x(t)  +  x(t-2)  + (5+4\delta_1)w(t)  + 1.5 w(t-1) +(\minus 3+\delta_1)w(t-2)
		\end{array}
		\right.
		\end{equation}
		where $\delta_1$ and $\delta_2$ are confined to $|\delta_1|\leq 0.2$ and $|\delta_2| \leq 0.3$, and its associated characteristic matrix:
		\begin{equation}
		\label{eq:illustration_link_1_SDEP}
		\begin{bmatrix}
		\lambda & 0 & 0 \\
		0 & 0 & 0 \\
		0 & 0 & 0
		\end{bmatrix}
		-
		\begin{bmatrix}
		\minus 2+\delta_1 & \minus 1 & 0 \\
		0  & \minus 1 & \Delta \\
		\minus 2+2\delta_2 & 5+4\delta_1 & \minus 1
		\end{bmatrix}
		-
		\begin{bmatrix}
		1+\delta_2 & 0 & 0 \\
		0 & 0 & 0 \\
		0 & 1.5 & 0  
		\end{bmatrix} e^{- \lambda}
		-
		\begin{bmatrix}
		0 & \minus 0.5+\delta_1 & 0\\
		0 & 0 & 0 \\
		1 & \minus 3 + \delta_1 & 0 
		\end{bmatrix} e^{-2 \lambda} \text{.}
		\end{equation}
		We start again with characterising the robust structured complex distance to instability. From Proposition~\ref{lem:expression_dist_nwp} it follows that 
		\[ \dist_{NWP}(\hat{\delta}) = \Big(\max_{|\delta_1| \leq 0.2} |5+4\delta_1|\Big)^{-1} = 1/5.8=0.1724 \text{.}\]
	As seen in Figure~\ref{fig:illustration_link1_psa}, the zero-crossing of $\alpha^{\mathrm{ps}}(\hat{\delta},\epsilon,\vec{\tau})$ lies to the right of the zero-crossing of $\alpha_{D,s}^{\mathrm{ps}}(\hat{\delta},\epsilon,\vec{\tau})$, which means that only the robust structured complex distance to a characteristic root chain crossing is finite:
		\[
	\begin{aligned}
	\dist_{CHAIN}(\hat{\delta}) &= 1/10.1 = 0.0990\\
	\dist_{FIN}(\hat{\delta}) &= +\infty \text{.}
	\end{aligned}
	\]
	The robust structured complex distance to instability is thus again equal to the robust structured complex distance to a characteristic root chain crossing. But unlike the previous example, all points in the $(\hat{\delta},\dist_{INS}(\hat{\delta}))$-pseudo-spectrum of \eqref{eq:illustration_link_1_SDEP} lie bounded away from the imaginary axis as $\alpha^{\mathrm{ps}}(\hat{\delta},\dist_{INS}(\hat{\delta}),\vec{\tau}) < 0$. However, $\alpha_{D,s}^{\mathrm{ps}}(\hat{\delta},\dist_{INS}(\hat{\delta}),\vec{\tau}) = 0$ implies that there exist perturbations on the delays that can be chosen arbitrarily small such that the spectrum of \eqref{eq:illustration_link_1_SDEP} contains a chain of characteristic roots with a vertical asymptote in the closed right-half plane for some $\Delta$ with $\|\Delta\|_2 \leq \dist_{INS}(\hat{\delta})$ and some $\delta \in \hat{\delta}$. This is illustrated in Figure~\ref{fig:illustration_link_1_spectrum}. Consider the following perturbations, which are associated with the loss of strong stability:
	\begin{center}
		$\Delta = 1/10.1$, $\delta_1 = 0.2$ and $\delta_2 = 0$.
	\end{center}
	 Figure~\ref{fig:illustration_link_1_spectrum_1} shows the spectrum for the associated realisation of \eqref{eq:illustration_link_1_SDEP} for the nominal delays. In this case all characteristic roots lie bounded away from the imaginary axis. Figure~\ref{fig:illustration_link_1_spectrum_2} shows its spectrum for a small perturbations on the delays. Now we have a chain of characteristic roots with the imaginary axis as vertical asymptote. Furthermore,  it can be shown that this vertical asymptote exists for all $\vec{\tau} = (1,2+\pi/n)$ with $n \in \mathbb{N}$.

	Next we establish the link with the robust strong $\mathcal{H}_{\infty}$-norm of system \eqref{eq:illustration_link_1_system}. Figure~\ref{fig:illustration_link1_supT} shows its worst-gain function. This function attains its maximum of $9.404$ (indicated in magenta) at a finite frequency ($\omega = 1.525$) as the robust asymptotic $\mathcal{H}_{\infty}$-norm equals (indicated in yellow):
	\[
	\|T_a(\cdot;\cdot,\vec{\tau}\|_{\mathcal{H}_{\infty}}^{\hat{\delta}} = \max_{|\delta_1|\leq 0.2} \max_{\omega \in \mathbb{R}} |5+4\delta_1+1.5e^{\jmath \omega}+(-3+\delta_1)e^{2\jmath \omega} |= 8.7477\text{.}
	\]
	 However the robust \emph{strong} asymptotic $\mathcal{H}_{\infty}$-norm (indicated in red in Figure~\ref{fig:illustration_link1_supT}) is equal to:
	 \[
	 \begin{aligned}
	 |||T_a(\cdot;\cdot,\vec{\tau})|||_{\mathcal{H}_{\infty}}^{\hat{\delta}} &= \max_{|\delta_1|\leq 0.2} \max_{\vec{\theta} \in [0,2\pi)^{2}} |5+4\delta_1+1.5e^{\jmath\theta_1}+(-3+\delta_1) e^{\jmath\theta_2}| \\ &= 10.1 \\ \big(&=\dist_{CHAIN}(\hat{\delta})^{-1}\big)
	 \end{aligned}\text{,}\] 
which means that the robust strong $\mathcal{H}_{\infty}$-norm corresponds to the robust strong asymptotic $\mathcal{H}_{\infty}$-norm.

		\begin{figure}
				\centering
				\includegraphics[width=0.6\linewidth]{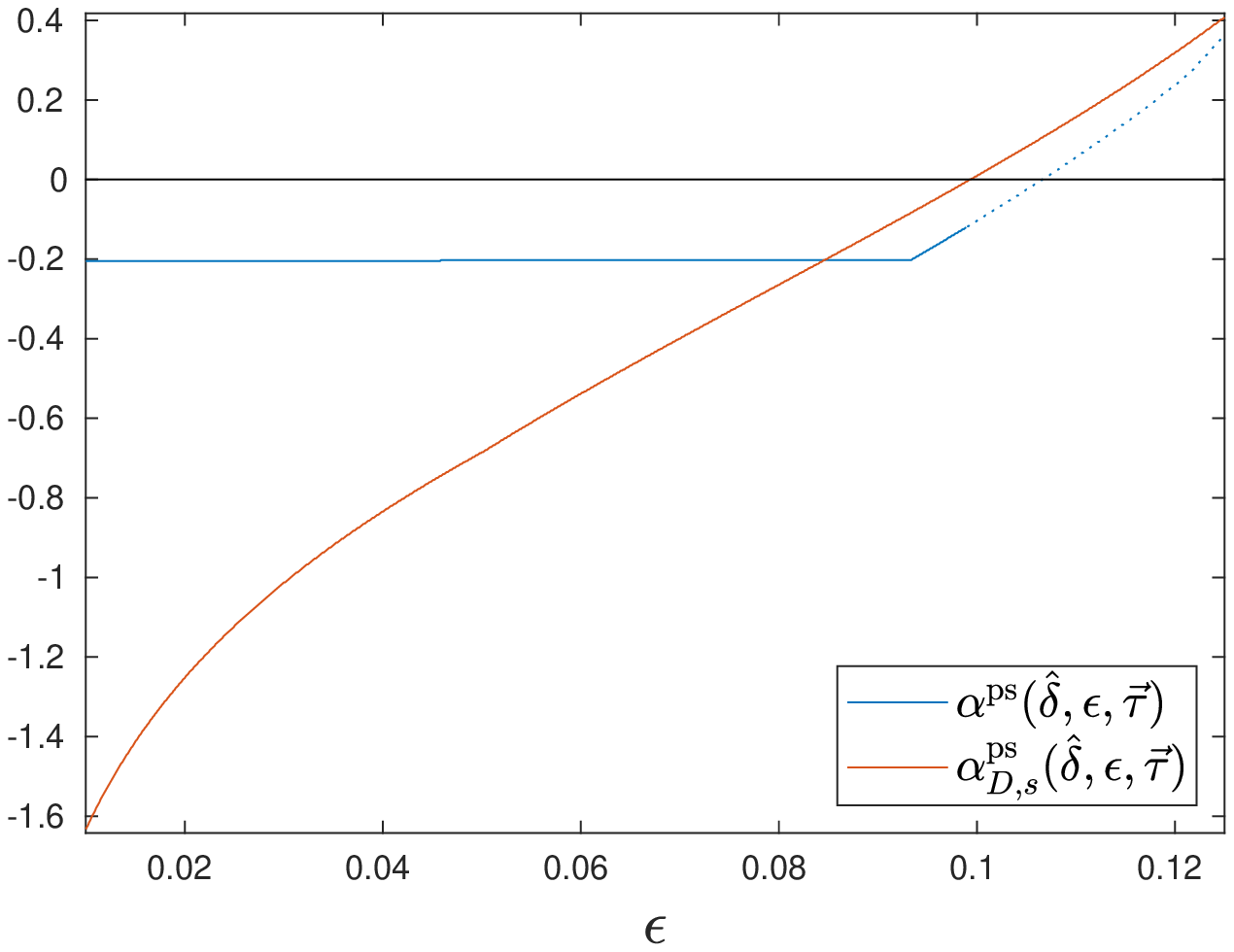}
				\caption{$\alpha^{\mathrm{ps}}(\hat{\delta},\epsilon,\vec{\tau})$ (blue) and $\alpha_{D,s}^{\mathrm{ps}}(\hat{\delta},\epsilon,\vec{\tau})$ (red) of characteristic matrix \eqref{eq:illustration_link_1_SDEP} in function of $\epsilon$.}
				\label{fig:illustration_link1_psa}
		\end{figure}
		\begin{figure}
		\centering
		\begin{subfigure}{0.48\linewidth}
		\centering
		\includegraphics[width=\linewidth]{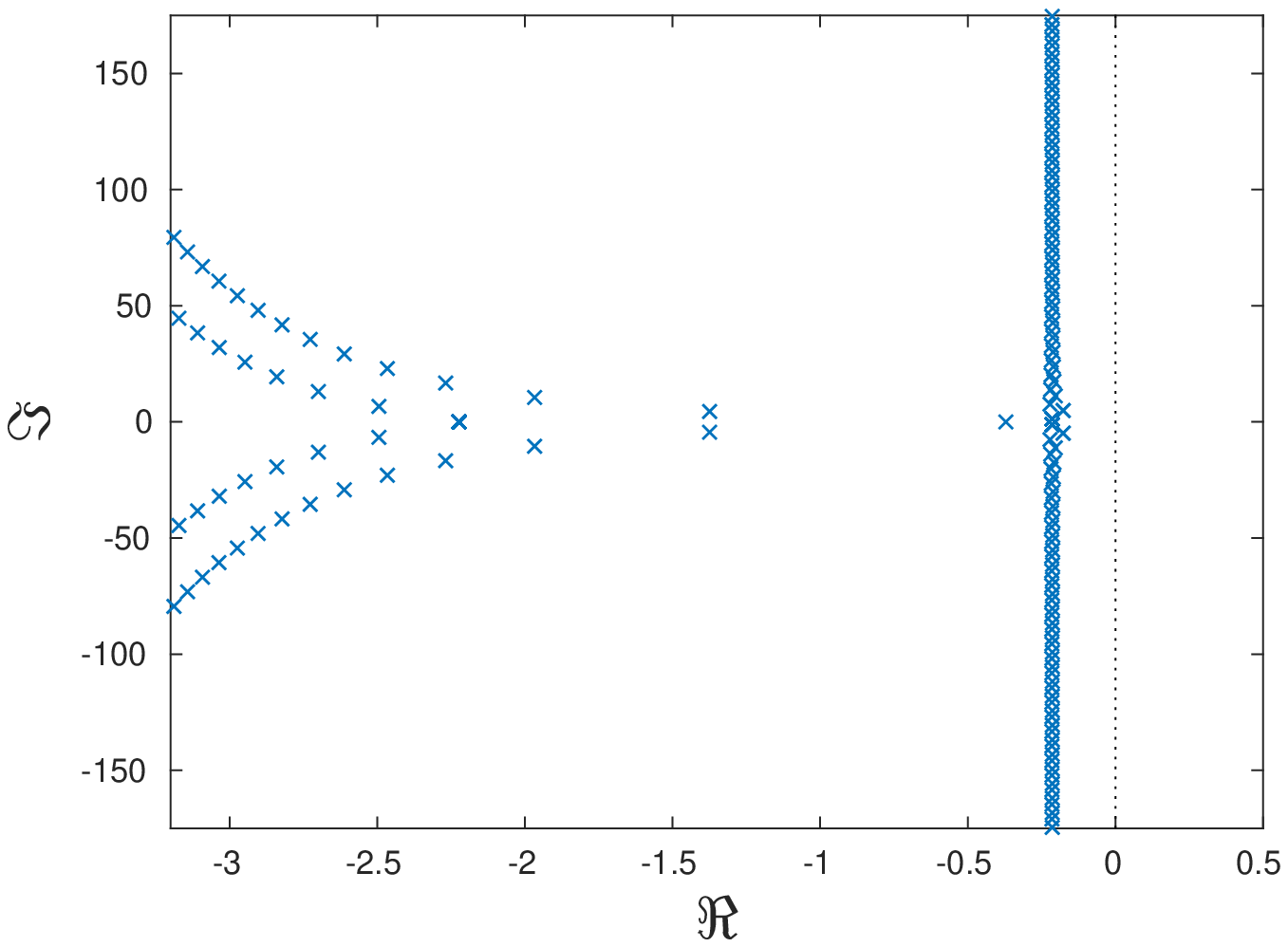}
		\caption{}
		\label{fig:illustration_link_1_spectrum_1}
		
		\end{subfigure}
		\begin{subfigure}{0.48\linewidth}
		\centering
		\includegraphics[width=\linewidth]{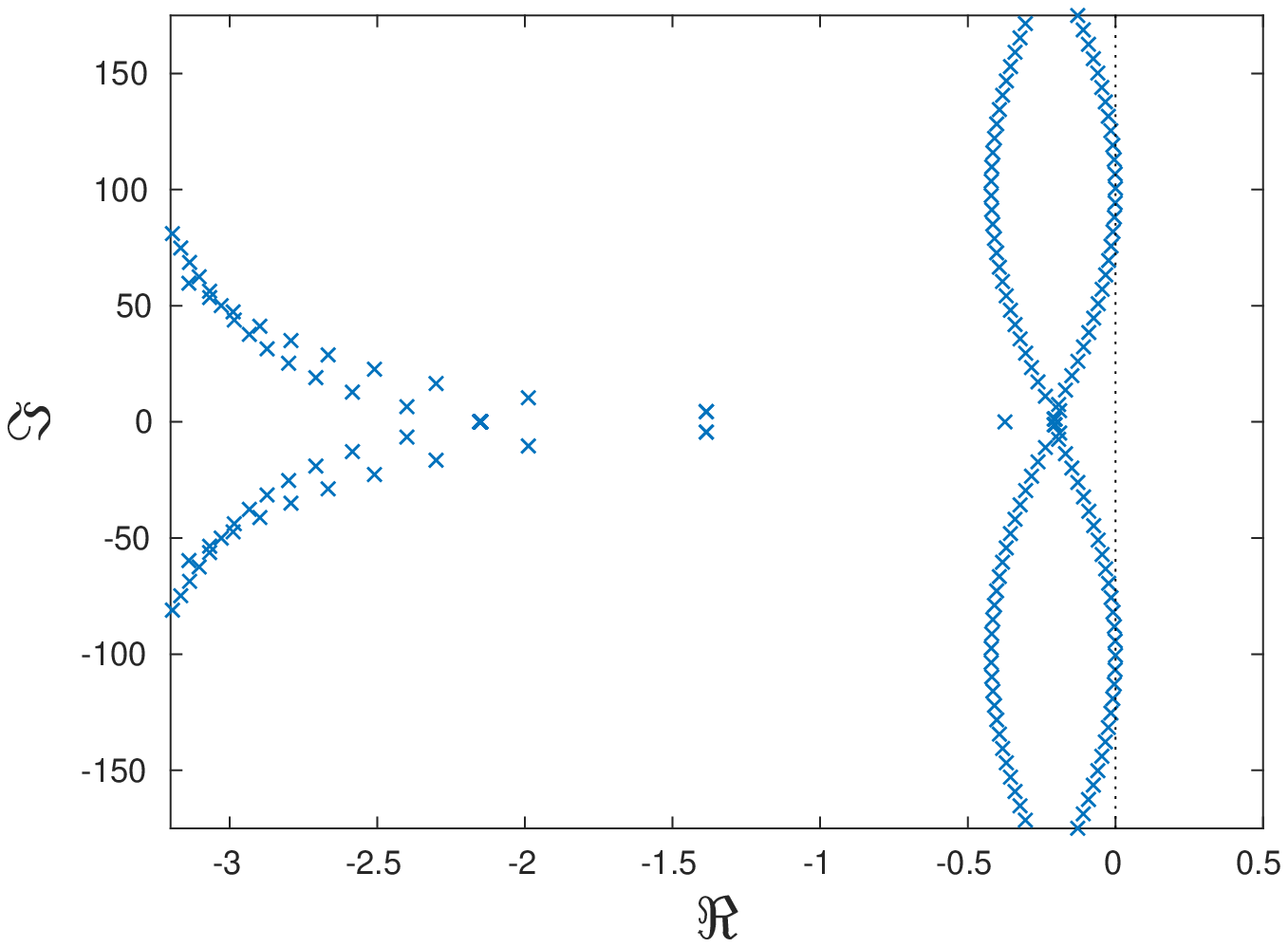}
		\caption{}
		\label{fig:illustration_link_1_spectrum_2}
		
		\end{subfigure}
		\caption{The spectrum of \eqref{eq:illustration_link_1_SDEP} for $\delta_1 = 0.2$, $\delta_2 = 0$, $\Delta = 1/10.1$ and $\vec{\tau} = (1,2)$ (left) and $\vec{\tau} = (1,2+\pi/100)$ (right)}
		\label{fig:illustration_link_1_spectrum}
		\end{figure}
		\begin{figure}[!hbtp]
			\centering
				\includegraphics[width=0.6\linewidth]{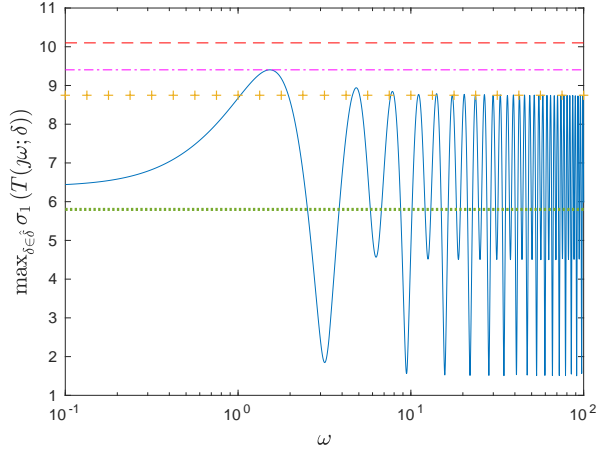}
				\captionsetup{width=.9\linewidth}
				\caption{The worst-case gain function of system~\eqref{eq:illustration_link_1_system} (blue), $||T(\cdot;\cdot,\vec{\tau})||_{\mathcal{H}_{\infty}}^{\hat{\delta}}$ (magenta dot-dashed), $\|T_a(\cdot;\cdot,\vec{\tau}\|_{\mathcal{H}_{\infty}}^{\hat{\delta}}$ (yellow +), $|||T_a(\cdot;\cdot,\vec{\tau})|||_{\mathcal{H}_{\infty}}^{\hat{\delta}}$ (red dashed) and $\dist_{NWP}(\hat{\delta})^{-1}$ (green dotted).}
				\label{fig:illustration_link1_supT}
		\end{figure}
	\end{example}
	In the previous examples we encountered three ways in which a characteristic matrix of form \eqref{eq:SDEP} can loose strong stability. In the first example the loss of strong stability was caused by a finite number of characteristic roots moving into the right-half plane. In this case the robust strong $\mathcal{H}_{\infty}$-norm of the associated system was equal to the maximum of the worst-case gain function. In the second example the loss of strong stability was caused by the asymptote of a chain of characteristic roots moving into the closed right-half plane. Now the robust strong $\mathcal{H}_{\infty}$-norm of the associated system was equal to the robust (strong) $\mathcal{H}_{\infty}$-norm of the asymptotic transfer function. In the last example the loss of strong stability was caused by the asymptote of a chain of characteristic roots moving into the closed half-plane, not for the nominal delay values but for infinitesimal delay perturbations. For this case the robust $\mathcal{H}_{\infty}$-norm and the robust strong $\mathcal{H}_{\infty}$-norm no longer coincided and the robust strong $\mathcal{H}_{\infty}$-norm was equal to the robust \emph{strong} asymptotic  $\mathcal{H}_{\infty}$-norm.
	\section{Numerical algorithm for computing the robust strong $\mathcal{H}_{\infty}$-norm}
	\label{sec:numerical_algorithm}
	This section introduces a high-level description of a numerical algorithm to compute the robust strong $\mathcal{H}_{\infty}$-norm of system \eqref{eq:system_description} using its relation with the robust structured complex distance to instability of characteristic matrix \eqref{eq:SDEP}.

	\begin{enumerate}
		\item[STEP 0] Check if Assumption~\ref{asm:robust_internally_stable} holds, ie. uncertain system \eqref{eq:system_description} is internally exponentially stable for all admissible perturbations, using the method presented in \cite{Borgioli2019}. 
		\item[STEP 1] Compute the robust strong asymptotic $\mathcal{H}_{\infty}$-norm by solving the following (constrained) optimisation problem:
		\[
		\begin{aligned}
		|||T_a(\cdot;\cdot,\vec{\tau})|||_{\mathcal{H}_{\infty}}^{\hat{\delta}} \
		=& \max_{\delta \in \hat{\delta}}\max_{\vec{\theta} \in [0,2\pi)^{K}} \sigma_1\left(\tilde{D}_0(\delta)+\sum_{k=1}^{K}\tilde{D}_k(\delta) e^{\jmath \theta_k}\right) \\
		\Big(=& \min\left\{\dist_{NWP}(\hat{\delta}), \dist_{CHAIN}(\hat{\delta})\right\}^{-1}\Big)\text{.}
		\end{aligned}
	\]
	Appendix~\ref{subsec:psafm_II} briefly explains how to solve this optimisation problem using the projected gradient flow method.

	\begin{remark}
	
	The robust structured complex distances to non well-posedness and a characteristic root chain crossing can also be computed separately, although by themselves they are not necessary to find the robust strong (asymptotic) $\mathcal{H}_{\infty}$-norm. The following expressions for these distance measures \big(for more general $Q$, $\tilde{P}_0(\delta,\Delta)$ and $\tilde{P}_k(\delta)$\big) follow from Section~\ref{sec:preliminaries}. The robust structured complex distance to non well-posedness is equal to the smallest $\epsilon$ for which the function 
		\[
		\mathbb{R}^{+} \ni \epsilon \mapsto \min_{\delta \in \hat{\delta}}\min_{\substack{\Delta \in \mathbb{C}^{m\times p}\\\|\Delta\|_2=\epsilon}}  \left\{\sigma_{\min}\left(U_{\mathcal{N}}^{H} \tilde{P}_0(\delta,\Delta) V_{\mathcal{N}}\right)\right\}\text{,}
		\]
		with $\sigma_{\min}(\cdot)$ the smallest singular value, equals zero. The robust structured complex distance to a characteristic root chain crossing is equal to the zero-crossing of 
		\[
		[0,\dist_{NWP}) \ni \epsilon \mapsto \max_{\delta \in \hat{\delta}} \max_{\substack{\Delta \in \mathbb{C}^{m\times p} \\ \|\Delta\|_2\leq\epsilon}} \max_{\vec{\theta}\in [0,2\pi)^{K}}\rho\left(\sum_{k=1}^{K} \left(U_{\mathcal{N}}^{H}\tilde{P}_0(\delta,\Delta)V_{\mathcal{N}}\right)^{\minus 1}U_{\mathcal{N}}^{H}\tilde{P}_{k}(\delta) V_{\mathcal{N}}e^{\jmath\theta_k} \right) \minus 1	
		\]
		where $\rho(\cdot)$ the spectral radius. 
		
		 In both cases one has find to find the zero(-crossing) of a function for which each function evaluation consists of solving an optimisation problem. This suggests a two-level approach: on the outer level  a root-finding method such as the Newton-bisection method, which combines the robustness of the bisection method with the fast (local) convergence of the Newton method (see \cite{PressWilliamH1996NriF} for a reference implementation), is used to find new estimates for the critical $\epsilon$; while on the inner level an optimisation method, such as the projected gradient flow method (see Appendix~\ref{sec:projected_steepest_ascend_flow}), is used to solve the (constrained) optimisation problem for a given $\epsilon$.
		\end{remark}
		
		\item[STEP 2] Compute the robust structured complex distance to finite root crossing by finding the zero-crossing of  \[
		\Big[0, \left(|||T_a(\cdot;\cdot,\vec{\tau})|||_{\mathcal{H}_{\infty}}^{\hat{\delta}}\right)^{-1} \Big) \ni \epsilon \mapsto \alpha^{\mathrm{ps}}(\hat{\delta},\epsilon,\vec{\tau})\text{,}
		\]
		with $\alpha^{\mathrm{ps}}(\hat{\delta},\epsilon,\vec{\tau})$ as defined in Remark~\ref{remark:pseudo_spectral_abscissa}. To find this zero-crossing once again a two-level approach is used. On the outer level the Newton-bisection method is used to find new estimates for $\epsilon$. While on the inner level the projected gradient flow method is used to compute $\alpha^{\mathrm{ps}}(\hat{\delta},\epsilon,\vec{\tau})$ for a given $\epsilon$.  The resulting flow and how to compute the derivative of $\alpha^{\mathrm{ps}}(\hat{\delta},\epsilon,\vec{\tau})$ with respect to $\epsilon$ (needed for the Newton-bisection method) will be outlined in Appendix~\ref{subsec:psafm_III}.

		\item[STEP 3] By Theorem \ref{the:main_theoretical_result} the robust strong $\mathcal{H}_{\infty}$-norm is equal to
		\[ 
		\begin{aligned}
		|||T(\cdot;\cdot,\vec{\tau})|||_{\mathcal{H}_{\infty}}^{\hat{\delta}} &= \min\left\{\dist_{NWP}(\hat{\delta}),\dist_{CHAIN}(\hat{\delta}),\dist_{FIN}(\hat{\delta})\right\}^{-1} \\
		&= \max\left\{|||T_a(\cdot;\cdot,\vec{\tau})|||_{\mathcal{H}_{\infty}}^{\hat{\delta}},\dist_{FIN}(\hat{\delta})^{-1}\right\} \text{.}
		\end{aligned}		
		\] 
	\end{enumerate}

	\section{Generalisations}
	\label{sec:generalisation}
	\subsection{Bounded uncertainties on delays}
	The presented theory and algorithm can easily be extended to systems with (bounded) uncertainties on both the coefficient matrices and the delays. Theorem \ref{the:main_theoretical_result} can be generalised to this case by incorporating the uncertainties on the delays in characteristic matrix family \eqref{eq:SDEP} and extending the definition of the robust structured complex distance to instability to also take these uncertainties into account. Furthermore as the robust strong asymptotic $\mathcal{H}_{\infty}$-norm, the robust structured complex distance to non well-posedness and the robust structured complex distance to a characteristic root chain crossing are independent of the delays, the proofs in Section~\ref{subsec:distance_nonwellposed} can be reused without modification. The results from Section~\ref{subsec:distance_finite_root} trivially generalise to the uncertain delay case by extending the definitions of the worst-case gain function and the robust structured complex distance to finite root crossing. Also the algorithm presented in Section~\ref{sec:numerical_algorithm} only need a minor modification: one has to include the uncertainties on the delays in the computation of the the pseudo-spectral abscissa in STEP 2.

	\subsection{System families described by delay-differential algebraic equations}
	The results can also be generalised to models described by (uncertain) delay-differential algebraic equations of the following form:
	\begin{equation}
	\label{eq:DDAE_system}
		\left\{
		\begin{array}{rcl}
		E \dot{x}(t) & = & \tilde{A}_0(\delta) x(t) + \sum_{k=1}^{K} \tilde{A}_k(\delta) x(t-\tau_k) + \tilde{B}_0(\delta) w(t) + \sum_{k=1}^K \tilde{B}_k(\delta) w(t-\tau_k) \\[1em]
		z(t) & = & \tilde{C}_0(\delta) x(t)+ \sum_{k=1}^K \tilde{C}_k(\delta) x(t-\tau_k) + \tilde{D}_0(\delta) w(t) + \sum_{k=1}^K \tilde{D}_k(\delta) w(t-\tau_k)
	\end{array}
	\right.		
	\end{equation}
	where the real-valued perturbations $\delta$ are confined to a specified set $\hat{\delta}$. In this formulation $x \in \mathbb{R}^{n}$ is the state vector, $w \in \mathbb{R}^{m}$ the exogenous input and $z \in \mathbb{R}^{p}$ the exogenous output, $E$ a real-valued, possibly singular, $n\times n$ matrix and $\delta$, $\hat{\delta}$, $\tilde{A}_k(\delta)$, $\tilde{B}_k(\delta)$, $\tilde{C}_k(\delta)$, $\tilde{D}_k(\delta)$ and $\tau_k$ as defined in Section~\ref{sec:introduction}. To avoid lack of causality and the occurrence of impulsive solutions, we assume that ${U_{E}}^{H}\tilde{A}_0(\delta)V_{E}$ is invertible for all $\delta \in \hat{\delta}$, with $U_E$ and $V_E$ $n \times \big(n-\rank(E)\big)$-dimensional matrices whose columns form a basis for respectively the left and right null space of $E$ \cite{Michiels2011}. 
	\begin{remark}
		We assume that the considered uncertain system has no uncertainties on $E$, as the matrix $E$ defines the structure of the differential and algebraic part of the equations and therefore typically does not contain parameters.
	\end{remark}
	
	Model class \eqref{eq:DDAE_system} can describe a wide variety of systems, even neutral systems can be reformulated in this form \cite{Michiels2011}. As a consequence the internal exponential stability of a realisation of the system~\eqref{eq:DDAE_system} is potentially sensitive to arbitrary small delay perturbations. Therefore we need to tighten Assumption~\ref{asm:robust_internally_stable} and assume that all admissible systems are \emph{strongly} internally exponentially stable. 
	
	Using similar derivations as in Section~\ref{sec:rob_strong_hinf_dist_instab}, it can be shown that the robust strong $\mathcal{H}_{\infty}$-norm of uncertain system \eqref{eq:DDAE_system} (under the aforementioned assumptions) is equal to the reciprocal of the robust structured complex distance to instability of \eqref{eq:SDEP} where the $Q$ matrix now has the following form:
	\[
	Q = \begin{bmatrix}
	E & 0 & 0 \\
	0 & 0 & 0 \\
	0 & 0 & 0
	\end{bmatrix} \text{.}
	\]

	Also the numerical algorithm presented in Section~\ref{sec:numerical_algorithm} can be extended to deal with uncertain delay-differential algebraic systems. In order to avoid the explicit computation of the robust structured complex distances to non well-posedness and a characteristic root chain crossing in STEP 1, an explicit expression for the asymptotic frequency response (and its  strong $\mathcal{H}_{\infty}$-norm) needs to be extracted first. Such expressions can be found in \cite[Equation~3.4 and Proposition~4.3]{Gumussoy2011}.
	\section{Examples}
	\label{sec:examples}
	An implementation of the algorithm described in Section~\ref{sec:numerical_algorithm} is available from \\ \url{http://twr.cs.kuleuven.be/research/software/delay-control/rb_hinf/}. To solve the constrained optimisation problems in steps 1 and 2 it uses the projected gradient flows presented in Appendix~\ref{sec:projected_steepest_ascend_flow}. The presented algorithm has also been validated on some test problems: Examples~\ref{example:finite_crossing},~\ref{example:chain_crossing}~and~\ref{example:strong_chain_crossing} in Section~\ref{subsec:central_theorem} and the open loop systems of the benchmark problems described in \cite[Section 7.3]{Gumussoy2011}\footnote{Available at \url{http://twr.cs.kuleuven.be/research/software/delay-control/hinfopt/}.} to which real-valued, structured uncertainties were added. These benchmark problems are available from the same location. 

	\section{Conclusion}
	\label{sec:conclusions}
In this paper we examined the relation between the robust (strong) $\mathcal{H}_{\infty}$-norm of a time-delay system with structured uncertainties and the robust structured complex distance to instability of an associated singular delay eigenvalue problem. We also introduced a novel numerical algorithm to compute this robust strong $\mathcal{H}_{\infty}$-norm. This robustness measure not only takes the considered perturbations on the system matrices into account, but  also infinitesimal perturbations on the delays. In this way a known fragility problem of the standard $\mathcal{H}_{\infty}$-norm, which might not continuously depend on the delay parameters, is eliminated. Theorem~\ref{the:main_theoretical_result} can be as seen as a extension of the well-known result by Hinrichsen and Pritchard that relates the {$\mathcal{H}_{\infty}$-norm} of a linear time-invariant system with the structured complex distance to instability of a perturbed eigenvalue problem \cite{Hinrichsen2005}, to systems with delays and real-valued uncertainties on the coefficient matrices.

In future work we plan to use the here presented method for the design of distributed controllers for interconnected networks of identical subsystems. As shown in \cite{Dileep2018}, for certain classes of networks this synthesis problem can be reformulated as a synthesis problem for a single subsystem with an additional parameter whose allowable values correspond to the spectrum of the adjacency matrix of the network.  By considering this parameter as an uncertainty that is bounded to a specified interval, the robust strong $\mathcal{H}_{\infty}$-norm of a single subsystem can be used to quantify the worst-case disturbance rejection of the complete network over all realisations of the network for which the eigenvalues are confined to this interval.  The here introduced algorithm to compute the robust strong H-infinity norm can thus be used as a building block of an algorithm for synthesizing robust controllers with favourable scalability properties in terms of the number of subsystems. 
\section*{Acknowledgements} This work was supported by the project C14/17/072 of the KU Leuven Research Council and by the project G0A5317N of the Research Foundation-Flanders (FWO - Vlaanderen). 
	\begin{appendices}
		\section{Projected gradient flow method}
		\label{sec:projected_steepest_ascend_flow}
		
		The projected gradient flow method is a continuous variant of the well-known steepest ascend/descend method for solving constrained optimisation problems. It looks for a flow, described by ordinary differential equations, along which the objective function monotonically increases/decreases. The flow is defined in such a way that the (local) optima of the objective function appear as attractive stationary points. These optimisers are found by discretising the flow (using for example Euler's forward method).
		
		There already exists an extensive literature \cite{Borgioli2019,Guglielmi2013,Guglielmi2014} on how to use the projected gradient flow method for computing extremal points of pseudo-spectra.  We will therefore restrict ourself to the resulting flows for the optimisation problems encountered in Section~\ref{sec:numerical_algorithm}. For more details we refer to the aforementioned papers. 
		
		\subsection{Step 1}
		\label{subsec:psafm_II}
		This subsection briefly describes how to use the projected gradient flow method for the optimisation problem encountered in STEP 1 of the algorithm described in Section~\ref{sec:numerical_algorithm}: 
		\begin{equation}
		\label{eq:maximisation_problem_step2}
		\begin{array}{cccc}
		|||T_a(\cdot;\cdot,\vec{\tau})|||_{\mathcal{H}_{\infty}}^{\hat{\delta}} =& 
		\underset{\delta,\vec{\theta}}{\text{maximise}}
		& & \sigma_{1}\left(\tilde{D}_0(\delta)+\sum_{k=1}^{K} \tilde{D}_k(\delta) e^{\jmath\theta_k} \right) \\
		&
		\text{subject to}
		& & \delta \in \hat{\delta} \\ 
		& & & \vec{\theta} \in [0,2\pi)^{K} \text{.}
		\end{array}
		\end{equation}
		To solve this maximisation problem we construct a path in the search space along which the objective function monotonically increases:
		\[
		\begin{aligned}
		\delta_l(\mathsf{t}) &= \bar{\delta}_l\delta^{n}_{l}(\mathsf{t}) \text{ with $\|\delta^{n}_{l}(\mathsf{t})\|_F\leq1$,} & l=1,\dots,L\\ 
		\theta_k(\mathsf{t}) &= \modolo(\vartheta_k(\mathsf{t}),2\pi), & k=1,\dots K 
		\end{aligned}
		\]
		with $\modolo(\cdot,\cdot)$ the modulo operator and 
		\[
		\left\{
		\begin{aligned}
		\dot{\vartheta}_k(\mathsf{t}) &= -\Im\left(u(\mathsf{t})^{H}\tilde{D}_k(\delta(\mathsf{t})) v(\mathsf{t}) e^{\jmath \theta_k(\mathsf{t})}\right)\\
		 \Xi_l(\mathsf{t})&= \bar{\delta}_l\sum_{k=0}^{K}\sum_{s=1}^{S_l^{D_k}}{G_{l,s}^{D_k}}^{T}\Re\left(u(\mathsf{t})v(\mathsf{t})^{H}e^{-\jmath\theta_k(\mathsf{t})}\right){H_{l,s}^{D_k}}^{T} \\
		\dot{\delta}^{n}_l(\mathsf{t}) & =  \begin{cases}
		\Xi_l(\mathsf{t}) - \Big\langle\delta^{n}_l(\mathsf{t}),\Xi_l(\mathsf{t})\Big\rangle_F \delta^{n}_l(\mathsf{t}) & \text{~if~} \|\delta^n_l(\mathsf{t})\|_F = 1 \text{~and~} \Big\langle\delta^{n}_l(\mathsf{t}),\Xi_l(\mathsf{t})\Big\rangle_F >0 \\
		\Xi_l(\mathsf{t}) &  \text{~otherwise~}
		\end{cases}
		\end{aligned}\right.
		\]
		where $u(\mathsf{t})$ and $v(\mathsf{t})$ are the left and right singular vectors (of unit norm) associated with the largest singular value of $\tilde{D}_0(\delta(\mathsf{t})) + \sum_{k=1}^{K}\tilde{D}_k(\delta(\mathsf{t}))e^{\jmath\theta_k(\mathsf{t})}$ and $\langle A,B \rangle_F = \sum_{i,j}A_{i,j}B_{i,j}$. Note that this path can be seen as the projection of the derivative of the largest singular value of $\tilde{D}_0(\delta)+\sum_{k=1}^{K} \tilde{D}_k(\delta) e^{\jmath\theta_k}$ with respect to respectively $\theta_k$ and the elements of $\delta_l$ onto the search space. The projection ensures that the constraints of the optimisation problem are fulfilled for all $\mathsf{t}$.
		\begin{remark} Optimisation problem \eqref{eq:maximisation_problem_step2} is highly non-convex (especially with respect to $\theta$). To improve the chance of finding the global optimum, one needs to restart the projected gradient flow method with several initialisations of the variables.
		\end{remark}
		\subsection{Step 2}
		\label{subsec:psafm_III}
		In this subsection we briefly describe the usage of the projected gradient flow method for the optimisation problem encountered in STEP 2 of the algorithm described in Section~\ref{sec:numerical_algorithm}: 
		\begin{equation}
		\label{eq:maximisation_problem_step3}
		\begin{array}{cccc}
		\alpha^{\mathrm{ps}}(\hat{\delta},\epsilon,\vec{\tau}) =  &
		\underset{\delta,\Delta}{\text{maximise}}
		& & \Re\big(\lambda_{RM}(\delta,\Delta)\big) \\
		& \text{subject to}
		& & \delta \in \hat{\delta} \\
		& & & \Delta \in \mathbb{C}^{m \times p} \\
		& & & \|\Delta\|_2 \leq \epsilon
		\end{array}
		\end{equation}
		with $\lambda_{RM}(\delta,\Delta)$ the right-most eigenvalue of $M(\lambda;\delta,\Delta,\vec{\tau})$.\\
		\begin{remark} The maximum of \eqref{eq:maximisation_problem_step3} might not be attained, as $M(\lambda;\delta,\Delta,\vec{\tau})$ might be neutral for $\Delta \neq 0$. To guarantee that the maximum of \eqref{eq:maximisation_problem_step3} is defined, we add an additional constraint to the optimisation problem: 
		\[\lambda_{RM}(\delta,\Delta) \in \left\{ \lambda \in \mathbb{C} : \Im(\lambda) \in \left[-\bar{\lambda},\bar{\lambda}\right]\right\}\] with $\bar{\lambda}$ sufficiently large. This may lead to an underestimate for $\alpha(\hat{\delta},\epsilon,\vec{\tau})$ for a given $\epsilon$. However for $\epsilon \in \left[0,\min\{\dist_{NWP}(\hat{\delta}),\dist_{CHAIN}(\hat{\delta})\}\right)$, the transition to a positive $(\hat{\delta},\epsilon)$-pseudo-spectral abscissa is caused by a (finite) characteristic root crossing the imaginary axis, and thus if $\bar{\lambda}$ is sufficiently large this additional constraint does not influence the result of the overall root finding procedure in STEP 2 of the algorithm.
		\end{remark}

		The following proposition allows us to restrict the search space for $\Delta$ and hence improve the computational efficiency.
		\begin{proposition}
			If $\lambda^{\star}$ does not lie in the ($\hat{\delta}$,0)-pseudo spectrum of \eqref{eq:SDEP} and is a (local) maximum of \eqref{eq:maximisation_problem_step3} for $\epsilon > 0$ with associated optimisers $\delta$ and $\Delta$, then there exists a rank 1-matrix $\Delta_{1} \in \mathbb{C}^{m\times p}$ with $\|\Delta_{1}\|_2 = \epsilon$ such that $\lambda^{\star}$ is preserved.
		\end{proposition}
		\begin{proof}
		It follows from Lemma~\ref{lem:relation_tf_ps_SDEP}, that
		\[
		\Lambda^{\mathrm{ps}}(\hat{\delta},\epsilon,\vec{\tau}) = \Lambda^{\mathrm{ps}}(\hat{\delta},0,\vec{\tau}) \bigcup \left\{s\in\mathbb{C}\setminus \Lambda^{\mathrm{ps}}(\hat{\delta},0,\vec{\tau}) : \max_{\delta \in \hat{\delta}}\sigma_1\left(T(s;\delta,\vec{\tau})\right) \geq \epsilon^{-1} \right\}
		\]
		with $\Lambda^{\mathrm{ps}}(\hat{\delta},\epsilon,\vec{\tau})$ as defined in \eqref{eq:combined_real_complex_pseudospectrum}. Furthermore, because $\lambda^{\star}$ is a (local) right-most point of the aforementioned pseudo-spectrum and is not in $\Lambda^{\mathrm{ps}}(\hat{\delta},0,\vec{\tau})$, it must lie in
		\[
		\left\{s\in\mathbb{C}\setminus \Lambda^{\mathrm{ps}}(\hat{\delta},0,\vec{\tau}) : \max_{\delta \in \hat{\delta}}\sigma_1(T(s;\delta,\vec{\tau})) = \epsilon^{-1} \right\}\text{.}
		\]
		By the second part of Lemma~\ref{lem:relation_tf_ps_SDEP} it follows that $\lambda^{\star}$ is a characteristic root of $M(\lambda;\delta,\Delta_{1},\vec{\tau})$ where $\Delta_{1} = \epsilon v u^{H}$ with $u$ and $v$ the (normalised) left and right singular vectors of $T(\lambda^{\star};\delta,\vec{\tau})$ associated with the singular value $\epsilon$.
		\end{proof}
		
		Based on this result, we define the following path, for which the optimizers of \eqref{eq:maximisation_problem_step3} appear as (attractive) stationary points:
		\[
		\begin{aligned}
		\delta_l(\mathsf{t}) &= \bar{\delta}_l \delta_l^{n}(\mathsf{t}) \text{ with $\|\delta_l^{n}(\mathsf{t})\|_F \leq 1$}, &l=1,\dots,L \\
		\Delta(\mathsf{t}) &= \epsilon u(\mathsf{t})v(\mathsf{t})^{H} \text{ with $\|u(\mathsf{t})\|_2 = \|v(\mathsf{t})\|_2 = 1$}
		\end{aligned}
		\]
		with
		\[
		\begin{aligned}
		\dot{u}(\mathsf{t}) &= \frac{\epsilon}{\xi(\mathsf{t})}\Big(\left(I-u(\mathsf{t})u(\mathsf{t})^{H}\right) R^{T}\phi(\mathsf{t})\psi(\mathsf{t})^{H}S^{T}v(\mathsf{t}) + \frac{\jmath}{2}\Im\left(u(\mathsf{t})^{H}R^{T}\phi(\mathsf{t}) \psi(\mathsf{t})^{H}Sv(\mathsf{t})\right)u(\mathsf{t})\Big)\\
		\dot{v}(\mathsf{t}) &= \frac{\epsilon}{\xi(\mathsf{t})}\Big(\left(I-v(\mathsf{t})v(\mathsf{t})^{H}\right)S\psi(\mathsf{t})\phi(\mathsf{t})^{H}Ru(\mathsf{t}) + \frac{\jmath}{2}\Im\left(v(\mathsf{t})^{H}S\psi(\mathsf{t}) \phi(\mathsf{t})^{H}Ru(\mathsf{t})\right)v(\mathsf{t})\Big)\\
		\Xi_l(\mathsf{t}) &= \frac{\bar{\delta}_l}{\xi(\mathsf{t})} \sum_{k=0}^{K}\sum_{s=1}^{S_l^k}{G_{l,s}^{k}}^{T}\Re\left(\phi(\mathsf{t})\psi(\mathsf{t})^{H}e^{-\overline{\lambda(\mathsf{t})}\tau_k}\right){H_{l,s}^{k}}^{T} \\
		\dot{\delta}^{n}_{l}(\mathsf{t}) &= \begin{cases}
		\Xi_l(\mathsf{t}) - \Big\langle\delta^n_l(\mathsf{t}),\Xi_l(\mathsf{t})\Big\rangle_F \delta_l^{n}(\mathsf{t}) & \text{~if~} \|\delta^{n}_l(\mathsf{t})\|_F = 1 \text{~and~} \Big\langle\delta^{n}_l(\mathsf{t}),\Xi_l(\mathsf{t})\Big\rangle_F >0 \\
		\Xi_l(\mathsf{t}) & \text{~otherwise~}
		\end{cases}
		\end{aligned}	
		\]
		with $\phi(\mathsf{t})$ and $\psi(\mathsf{t})$ the left and right eigenvectors associated with $\lambda_{RM}(\delta(\mathsf{t}),\Delta(\mathsf{t}))$ normalised such that $\xi(\mathsf{t}) = \phi(\mathsf{t})^{H}\left(Q+\sum_{k=1}^{K}\tilde{P}_k(\delta(\mathsf{t}))\tau_{k}e^{-\tau_k \lambda_{RM}(\delta(\mathsf{t}),\Delta(\mathsf{t}))})\right)\psi(\mathsf{t})$ is real and positive.
		\begin{remark}
			This path can be seen as a combination of the results in \cite{Guglielmi2013} and \cite{Borgioli2019}.
		\end{remark}
		\begin{remark}
			The right-hand sides of the last two equations can be interpreted as the projection of the derivative of $\lambda_{RM}(\delta,\Delta)$ with respect to the elements of $\delta^{n}_l$ on the search space. The projection assures that the norm constraint on $\delta^{n}_l(\mathsf{t})$ is fulfilled for all $\mathsf{t}$.
		\end{remark}

		To use the Newton-bisection method in STEP 2 of the algorithm described in Section~\ref{sec:numerical_algorithm}, one requires both the $(\hat{\delta},\epsilon)$-pseudo-spectral abscissa and its derivative with respect to $\epsilon$. The latter can be obtained cheaply from the optimizers of optimisation problem \eqref{eq:maximisation_problem_step3}: let $\delta^{\star}$ and $\Delta^{\star} = \epsilon u^{\star} {v^{\star}}^{H}$ be the maximizers of optimisation problem \eqref{eq:maximisation_problem_step3} and if $\lambda^{\star} = \lambda_{RM}(\delta^{\star},\Delta^{\star})$ is simple with corresponding left and right eigenvectors $\phi^{\star}$ and $\psi^{\star}$, normalised such that ${\phi^{\star}}^{H}(Q+\sum_{k=1}^{K}\tilde{P}_k(\delta^{\star})\tau_{k}e^{-\tau_k \lambda^{\star}})\psi^{\star}$ is real and positive, then
		\[
		\dfrac{d\alpha^{\mathrm{ps}}(\hat{\delta},\epsilon,\vec{\tau})}{d\epsilon} = \dfrac{\Re\left({\phi^{\star}}^{H}R u^{\star} {v^{\star}}^{H} S\psi^{\star}\right)}{{\phi^{\star}}^{H}(Q+\sum_{k=1}^{K}\tilde{P}_k(\delta^{\star})\tau_{k}e^{-\tau_k \lambda^{\star}})\psi^{\star}}\text{,}
		\] 
	 see \cite{Borgioli2019}.

\end{appendices}
	\bibliographystyle{unsrt}
	\bibliography{reference.bib}
\end{document}